\documentclass[a4wide,11pt]{article}
\usepackage{amsmath,latexsym, amsfonts, amssymb}

\usepackage{amssymb, amsmath, amsfonts}
\usepackage{color}
\usepackage{epsfig,amsbsy,amsthm}
\usepackage{float,epsfig}
\usepackage{graphics}
\usepackage{appendix}
\usepackage{algpseudocode}
\usepackage{algorithm}

\newtheorem{theorem}{Theorem}[section]
\newtheorem{lemma}{Lemma}[section]

\newtheorem{corollary}[theorem]{\bf Corollary}
\newtheorem{remark}{Remark}[section]
\newtheorem{definition}{Definition}[section]
\theoremstyle{definition}
\newtheorem{example}{Example}[section]
\newcommand {\mat}[1] {\left[\begin{array}{#1}}
\newcommand {\rix}{\end{array}\right]}
\def \R{{\mathbb R}}
\def \C{{\mathbb C}}

\def \sig{\sigma}

\def \C{{\mathbb C}}

\def \lam{\lambda}
\def \sig{\sigma}

\def \diag{\mathrm{diag}}

\def \rar{\rightarrow}

\newcommand{\eproof}{\space
    {\ \vbox{\hrule\hbox{\vrule height1.3ex\hskip0.8ex\vrule}\hrule}}\par}
\title{Fiedler Linearizations of Rectangular Rational Matrix Functions}
\author{Namita Behera \footnotemark[2]~\footnotemark[1]
\and Avisek Bist\footnotemark[4]
\and Volker Mehrmann\footnotemark[3]~\footnotemark[1]
}

\begin{document}
\maketitle
\begin{abstract} Linearization is a standard approach in the computation of eigenvalues, eigenvectors and invariant subspaces of matrix polynomials and rational matrix valued functions.  An important source of linearizations  are the so called  \emph{Fiedler linearizations}, which are generalizations of the classical companion forms.  In this paper  the concept of Fiedler linearization is extended from square regular to rectangular rational matrix valued functions. The approach is applied to Rosenbrock functions arising in mathematical system theory.
\end{abstract}
\noindent



\renewcommand{\thefootnote}{\fnsymbol{footnote}}
\footnotetext[3]{
Institut f\"ur Mathematik, MA 4-5, TU Berlin, Str. des 17. Juni 136,
D-10623 Berlin, FRG.
\texttt{mehrmann@math.tu-berlin.de}.
}
\footnotetext[2]{ Department of Mathematics, Sikkim University, Sikkim-737102, India
   \texttt{nbehera@cus.ac.in}.
   Partially supported IMU-ESB program 
}
\footnotetext[1]{
Partially supported by the {\it Einstein Stiftung Berlin} through
the  Research Center {\sc Matheon} {\it Mathematics for key technologies}
in Berlin.
}

\footnotetext[4]{ Department of Mathematics, Sikkim University, Sikkim-737102, India
   \texttt{abist.21pdmt01@sikkimuniversity.ac.in}
   }
\renewcommand{\thefootnote}{\arabic{footnote}}

\section{Introduction}
Linearization is one of the standard techniques in the eigenvalue problem for matrix polynomials and rational matrix valued functions. In the last decade a variety of linearization methods have been developed in order to deal with algebraic structures, with ill-conditioning in the representation, and in order to construct efficient numerical methods. An important source of such linearizations for matrix polynomials are the so called  \emph{Fiedler pencils}, which are generalizations of the companion forms \cite{DeTDM09}. These linearizations have been extended to regular rational matrix functions in \cite{DopMQV20,PerQ22}. For singular rational matrix functions, i.e. rational matrix valued functions which are non-square or do not have full normal rank, linearization requires new techniques which we will derive in this paper. This may be used in the  \emph{rational eigenvalue  problem},
\begin{equation}\label{ratevp}
R(\lam)u=0,
\end{equation}
where $R(\lambda) \in \mathbb F(\lam)^{p,m}$, the set of $p\times m$ matrices with entries that are rational functions with coefficients in the field $\mathbb F$, where $\mathbb F=\mathbb R$ or $\mathbb F=\mathbb C$ denotes the field of real or complex numbers, respectively.  Defining the \emph{ normal rank} of $R\in \mathbb F(\lam)^{p,m}$ as the rank in the  field of rational matrix valued functions, we say that  $\lambda_0$ is an \emph{eigenvalue} of $R(\lam)$ if the rank of $R(\lambda_0)$ is smaller than the normal rank.

Rational eigenvalue problems arise in many applications,
see e.g. \cite{AhmM16,BetHMST13,ConPV89,HilMM04,HwaLWW04,MehS11,MehV04,Pla82,Sol06,Vos03} and the references therein. We consider the special situation that  
the rational matrix valued function has a representation
\begin{equation}\label{trfun}
R(\lam) := D(\lam) + C(\lam) A(\lam)^{-1} B(\lam)  \in \C(\lam)^{p, m},
\end{equation}
where,  denoting by $\mathbb C[\lam]^{p, m}$ the vector space of $p \times m$ matrix polynomials with coefficients in $\mathbb C$, we assume that    $ A(\lam) \in \mathbb{C}[\lam]^{n,n}$, $B(\lam) \in \mathbb{C}[\lam]^{n, m}$, $ C(\lam) \in \mathbb{C}[\lam]^{p, n}$, $D(\lam) \in \mathbb{C}[\lam]^{p,m}$, and $A(\lam)$ is invertible as a polynomial matrix, i.e., $\det A(\lambda)$ does not vanish identically for all $\lambda$; in the following we call such a matrix polynomial \emph{regular}.

Let $R(\lam)\in \mathbb{C}(\lam)^{p, m}$ be a rational matrix function and let 
\begin{equation*}\label{smform} 
\mathbf{SM}(R(\lam)) = \diag\left( \frac{\phi_{1}(\lam)}{\psi_{1}(\lam)}, \cdots, \frac{\phi_{k}(\lam)}{\psi_{k}(\lam)}, 0_{p-k, m-k}\right)    
\end{equation*} 
be the Smith-McMillan form~\cite{Ros70} of  $R(\lam),$
where the scalar polynomials $\phi_{i}(\lam)$ and $ \psi_{i}(\lam)$  are monic, pairwise coprime, and  $\phi_{i}(\lam)$ divides $\phi_{i+1}(\lam)$ and $\psi_{i+1}(\lam)$ divides $\psi_{i}(\lam),$ for $i= 1, 2, \ldots, k-1$.
The polynomials $\phi_1(\lam), \ldots, \phi_k(\lam)$ and $ \psi_1(\lam), \ldots, \psi_k(\lam)$ are called {\em invariant zero polynomials} and {\em invariant pole polynomials} of $R(\lam),$ respectively. Define
$\phi_{R}(\lam) := \prod _{j=1}^{k} \phi_{j}(\lam) \,\,\, \mbox{ and } \,\,\, \psi_{R}(\lam) := \prod _{j=1}^{k} \psi_{j}(\lam).$
 A complex number $ \lam $ is said to be a  zero of $R(\lam)$ if $ \phi_R(\lam) =0$ and
a complex number $ \lam$ is said to be a  pole of $R(\lam)$ if $\psi_R(\lam) =0.$

Rational matrix valued functions of the form \eqref{trfun} arise, in particular, as transfer functions in linear system theory, see e.g. \cite{Var91}. Consider a linear time invariant system on the positive half line $\R_{+}$ in the representation
\begin{eqnarray}
 0&=& A(\frac{d}{dt}) x(t) -B(\frac{d}{dt}) u(t), \nonumber \\
 y(t)  &=& C(\frac{d}{dt}) x(t) + D(\frac{d}{dt}) u(t).\label{lti-system}
\end{eqnarray}
Here, the function $u : \R_{+} \rightarrow \R^{m}$ is the input vector, $x: \R_{+} \rightarrow \R^{n}$ is the state vector, $y : \R_{+} \rightarrow \R^{p}$ is the output vector, and for $Z(\lam) =\sum_{i=0}^d Z_i \lambda^i\in \mathbb{C}[\lam]^{p, m}$ we use  $Z(\frac{d}{dt})$ to denote the differential operator $\sum_{i=0}^\ell Z_i \frac{d^i}{dt^i}$, where $\frac{d}{dt}$ denotes time-differentiation.
Performing a Laplace transformation of the system, assuming zero initial conditions, and denoting by $\hat z$ the Laplace transform of a function $z(t)$,
one obtains a linear system 
\begin{equation} \label{polsys}
\mat{c}0 \\ \hat y\rix=  S(\lam) \mat{c} \hat x\\ \hat u\rix,
\end{equation}
with matrix polynomial coefficient
\begin{equation}\label{rosmatrix}
 S(\lam ) := \left[
                            \begin{array}{c|c}
                              A(\lam) & -B(\lam) \\
                              \hline
                              C(\lam) & D(\lam) \\
                            \end{array}
                          \right] \in \C[\lam]^{(n+p),(n+m)}.
\end{equation}
If $A(\lambda)$ is regular, then one can eliminate $\hat x$ and obtains $\hat y= R(\lambda) \hat u$ with $R(\lambda)$ as in \eqref{trfun}, and $\hat x= A(\lambda)^{-1}B(\lambda) \hat u$, , see \cite{Ros70}.
 In this case  $S(\lambda)$ is called a \emph{Rosenbrock system matrix}. We prefer the terminology \emph{Rosenbrock System Matrix Polynomial (RSMP)}.

Conversely, if one has a given  rational matrix function of the form
(\ref{trfun}), then one can always interpret it as originating
from an RSMP of the form (\ref{rosmatrix}). Such rational matrix valued functions arise e.g. from realizations of frequency domain input-output data, see e.g. \cite{MayA07}, or in model order reduction, see e.g. \cite{Ant05,AntBG10,BauBF14,GugA04}.

In order to understand the system behavior, one typically considers the polynomial eigenvalue problem
\begin{equation}\label{rosevp}
 S(\lam) \mat{c} \xi \\ \eta \rix= \mat{c|c}
                              A(\lam) & -B(\lam)\\
                              \hline
                              C(\lam) & D(\lam) \rix \mat{c} \xi \\ \eta \rix =0.
\end{equation}

If $A(\lam)$ is regular, then one can split this eigenvalue problem in the rational eigenvalue problem 
\begin{equation}\label{evpR}
R(\lambda) \eta=0,  
\end{equation}
and, for an  eigenvalue/eigenvector pair  $(\lambda_0,\eta_0)$  of \eqref{evpR}, the linear system $A(\lambda_0) \xi= B(\lambda_0)\eta_0$.

From this construction it is clear that the eigenvalues of  $S(\lambda)$ are the eigenvalues of $R(\lambda)$ and that the eigenvalues of $A(\lambda)$ are the poles of $R(\lambda)$. 

Note that $p=m$, i.e. that $D$ is a square matrix polynomial, is a necessary  condition for \eqref{evpR} to be a regular rational matrix function. In this case one can apply the recently developed linearization techniques of \cite{DopMQV20,PerQ22}. 
For non-square $D$ the situation is more involved and to discuss this question is the main topic of this paper, in which we 
assume for simplicity that $B,C$ are constant matrices in $\lambda$. However, all the results can be extended (with extra technicalities) to the case that $B,C$ depend on $\lambda$. 

The polynomial eigenvalue problem (\ref{rosevp}) can also be
considered for the case of non-square or singular matrix polynomials $A(\lambda)$. 
In this case however, first a regularization procedure is necessary, which deflates the singular part. This topic has been studied in the system theoretic context in \cite{BinMMS15,CamKM12,KunM06,Sch11} and in the context of structured eigenvalue problems in \cite{ByeMX08}. However, we do not discuss this regularization procedure in this paper.

The standard approach  to  solve a general rational eigenvalue problem of the form (\ref{ratevp}) is to first clear out the denominator, by multiplying with a scalar polynomial associated with all poles of the entries of $R$. In the case of (\ref{trfun}) one can e.g.   multiply by $\det A(\lambda)$, or what is computationally more practical, to study the eigenvalue problem directly for  $S(\lambda)$.

Both approaches transform the rational problem into a \emph{polynomial eigenvalue problem} but they may lead to slightly  different results. Consider the following example from \cite{AhmM16}.
\begin{example}\label{ex1}{\rm
The rational eigenvalue problem
\[
R(\lambda)\eta:=\mat{cc}
\lambda-2 +\frac{1}{\lambda-1} &1 \\ 1 & 0
\rix \eta=0,
\]
can be considered as the transfer function associated with the  RSMP
\[
 S(\lambda) =\mat{c|cc} \lambda -1 & 1 & 0 \\
\hline
-1 &\lambda -2 & 1 \\
0 & 1 & 0 \rix.
\]
In this case $R(\lambda)$ has no finite eigenvalues and a double and defective eigenvalue
at $\infty$, while
$S(\lambda)$ has the same double and defective eigenvalue at $\infty$ and as well as the  simple finite eigenvalue $1$.

Scaling $R(\lambda)$ by  $\lambda -1$ to clear out the denominator,  the rational eigenvalue problem
is turned into the quadratic eigenvalue problem
\[
P(\lambda)x=  \mat{cc}
(\lambda-1)(\lambda-2) +1 &\lambda-1 \\ \lambda-1 & 0
\rix x=0,
\]
which still has the double and defective eigenvalue at $\infty$ but now also a double and defective eigenvalue at $1$.

Scaling out the denominator thus has created an extra defective eigenvalue which makes the problem more (than already) sensitive to perturbations, \cite{DemK86,SteS90}, while when considering the RSMP  $S(\lambda)$} this extra eigenvalue is not arising.
\end{example}

Once a rational problem has been transformed into a polynomial problem, to compute eigenvalues, eigenvectors, and invariant subspaces, one can transform the system further by converting it to an equivalent generalized linear eigenvalue problem.  There are again many different ways to perform this \emph{linearization} process. Their advantages and disadvantages with respect to conditioning (sensitivity of eigenvalues under perturbations), structure preservation (respecting symmetry or group structures), sparsity (number or zero blocks or band structure), preservation of algebraic invariants (length of Jordan chains, multiplicities etc.)  have received a lot of attention in recent years, see e.g. \cite{AntV04,Beh14,BueD11,ByeMX08,DasAlm19,DeTDM09,DeTDP13,GohLR82,GraHT11,HigLT07,HigMMT06,HigMT06,HigMT09,MacMMM06b,MacMMM06a,Tis00,VolA11}.

Let $D(\lam)=\sum_{j=0}^{d_D}\lambda^j D_j \in \C[\lam]^{p, m}$, $C\in \C^{p,  n}$, $B \in \C^{n, m}$, let $A(\lam)= \sum_{i=0}^{d_A} \lambda^{i}A_i\in {\mathbb C}[\lambda]^{n,n}$ be regular and consider the associated  rectangular RSMP  $S(\lam)$ and  the transfer function $R(\lam)$. Then the eigenvalues of $R(\lam)$ can be computed by solving the generalized eigenvalue problem for the pencil \emph{first}  and \emph{second companion form} of  $S(\lam)$, namely
\begin{equation}
\mathcal{C}_1(\lam) := \lam X+Y ,\label{firstcomp}
\end{equation} with
{\scriptsize \begin{eqnarray*}
 X &:=& \left[
    \begin{array}{cccc|cccc}
      A_{d_A} &  &  &  &  & & & \\
       & I_n &  &  &  & & & \\
       &  & \ddots &  &  & & & \\
       &  &  & I_n &  & & & \\
      \hline
       &  &  &  & D_{d_D} & & & \\
       &  &  &  &  & I_m & & \\
       &  &  &  &  & & \ddots & \\
       &  &  &  &  &  & & I_m \\
    \end{array}
  \right],\\
Y &:=& \left[\begin{array}{cccc|cccc}
A_{d_A-1} & A_{d_A-2} & \cdots & A_0 & 0 & \cdots & 0 & -B \\
   -I_n & 0 & \cdots & 0 &  & 0 &  & 0 \\
    & \ddots & \ddots & & &  & \ddots & \vdots \\
    &  & -I_n & 0 &  &  &  & 0 \\
    \hline
    0 & \cdots & 0 & C & D_{d_D-1} & D_{d_D-2} & \cdots & D_0 \\
    & 0 &  & 0 & -I_m & 0 & \cdots & 0 \\
    &  & \ddots & \vdots &  & \ddots & \ddots & \vdots \\
    &  &  & 0 &  &  & -I_m & 0 \\
    \end{array}
\right],
\end{eqnarray*}}
respectively, \begin{equation}
\mathcal{C}_2(\lam) = \lam X_2 +Y_2 \label{secondcomp}    
\end{equation} with
\begin{eqnarray*}
&& X:=  \left[
                         \begin{array}{cccc|cccc}
                           A_{d_A} &  &  &  & & & & \\
                            & I_{n} &  &  & & & & \\
                            &  & \ddots &  & & & & \\
                            &  &  & I_{n} & & & & \\
                            \hline
                            &  &  &  & D_{d_D} & & & \\
                            &  &  &  &  & I_p & & \\
                            &  &  &  &  & & \ddots & \\
                            &  &  &  &  & &  & I_p \\
                         \end{array}
                       \right],\nonumber\\
                       &&  Y:= \left[
                \begin{array}{cccc|cccc}
                  A_{d_A-1} & -I_{n} &  &  & 0 & & & \\
                  A_{d_A-2} & 0 & \ddots &  & \vdots & & & \\
                  \vdots & \ddots &  & -I_{n} & 0 & & \ddots & \\
                  A_{0} & \cdots & 0 & 0  & -B &  0 & \cdots & 0 \\
                  \hline
                  0  &  &  &   &  D_{d_D-1} & -I_p & \cdots & 0 \\
                  \vdots  &  \ddots  &  &   &  D_{d_D-2} & & & \vdots\\
                  0  &  &   \ddots &   &  \vdots & & & -I_p\\
                  C  & 0 & \cdots &  0 &  D_0 & 0 & \cdots & 0\\
                \end{array}
              \right].\nonumber
\end{eqnarray*}

For rational eigenvalue problems, in \cite{DopMQV20} the problem is considered as a multivariable state space system and a definition for local linearizations of rational matrices is presented. This definition allows  to introduce matrix pencils associated to a rational matrix function that preserve its structure of zeros and poles in subsets of any algebraically closed field and also at infinity.
Recently, in  \cite{AmpDMZ18, DMQD22, DMQD23, PerQ22}, different linearizations of  the RSMP  $S(\lambda)$ in  (\ref{rosmatrix}) were studied. In \cite{AmpDMZ18}    strong linearizations of arbitrary rational matrices are discussed,  as well as characterizations of  linear matrix pencils for general rational matrices which are explicitly constructed from a minimal state-space realization of   a rational matrix. 
For these linearizations of arbitrary rational matrices also the  recovery of eigenvectors has been analyzed when  $S(\lambda)$ is regular, as well as minimal bases and minimal indices, when  $S(\lambda)$ is singular, see \cite{AmpDMZ18,AmpDMZ21}. 
Also, in \cite{DMQD23}  a new family of linearizations of rational matrices,  called block full rank linearizations has been studied. It is not clear from \cite{AmpDMZ18, DMQD22, DMQD23} whether the  Fiedler pencils of  $S(\lambda)$ which are studied in this paper are block permutationally similar to block minimal pencils.

It is also not at all clear whether the linearizations studied in \cite{AmpDMZ18, DMQD23, DMQD22, PerQ22}  include our Fiedler linearizations of RSMP. Therefore,  we study explicitly the Fiedler linearizations (which indeed are easy to construct from the coefficients of a rational matrix) of rectangular rational matrices associated to multivariable state-space systems. We also study the relationship between the eigenvalues of a rational eigenvalue problem given either in the form of a transfer function (\ref{trfun}), or its polynomial representation as an RSMP,  and different associated linearizations. This problem has been studied for square rational matrix functions in \cite{Beh14,SuB11} and we will extend these results to the rectangular case. 

The paper is organized as follows. In Section~\ref{sec:SQFP}  we recall the definition and some properties of Fiedler pencils for square RSMPs from \cite{AlaB16,Beh14,BehB22}. In Section \ref{sec:RFP} we define Fiedler pencils for rectangular RSMPs and in Section~\ref{sec:LFP} we prove that these are really linearizations for rectangular system matrices. Section~\ref{sec:concl} presents conclusions.

\section{Fiedler pencils for square Rosenbrock system matrix polynomials}\label{sec:SQFP}
In this section we recall the linearization of square RSMPs  with $B,C$ constant in $\lambda$, i.e. for RSMPs of the form
\begin{equation} \label{smpq}
 S(\lambda) = \left[ {\begin{array}{c|c}
A(\lambda) & -B \\
\hline
C & D(\lambda) \\
\end{array}}\right] \in {\mathbb C}[\lambda]^{n+m,n+m}
\end{equation}
and the associated transfer function
\[
R(\lambda) = D(\lambda) + C A(\lambda)^{-1}B \in {\mathbb C}(\lambda)^{m,m},
 \]
where $A(\lambda) = \sum_{i=0}^{d_A} \lambda^{i}A_i\in {\mathbb C}[\lambda]^{n,n}$  is a regular matrix polynomial of degree $d_A$ and
$D(\lambda) = \sum_{j=0}^{d_D}\lambda^j D_j\in {\mathbb C}[\lambda]^{m,m}$ is of degree $d_D$.

To these two matrix polynomials $A(\lam)$ and $D(\lam)$, we may apply the well-known linearization technique of \cite{DeTDM09}, via  \emph{Fiedler matrices}  $M_{i}$, $i= 0, 1, \ldots, d_A$, associated with $A(\lam)\in {\mathbb C}[\lambda]^{n,n}$, defined by
\begin{eqnarray}\label{imfp}
M_{d_A} &:=& \left[
          \begin{array}{cc}
            A_{d_A} &  \\
             & I_{(d_A-1)n} \\
          \end{array}
        \right],\ M_{0} := \left[
                   \begin{array}{cc}
                     I_{(d_A-1)n} &  \\
                      & -A_{0} \\
                   \end{array}
                 \right], \nonumber
\\
M_{i} &:=& \left[
  \begin{array}{cccc}
    I_{(d_A-i-1)n} &  & & \\
     & -A_{i} & I_{n} & \\
     & I_{n} & 0  &  \\
     &   &   & I_{(i-1)n}\\
  \end{array}
\right], \,\, i= 1, \ldots, d_A-1, 
\end{eqnarray}
and those associated with $D(\lam)\in {\mathbb C}[\lambda]^{m,m}$  defined by
\begin{eqnarray}
N_{d_D} &:=& \left[
          \begin{array}{cc}
            D_{d_D} &  \\
             & I_{(d_D-1)m} \\
          \end{array}
        \right],\ N_{0} := \left[
                   \begin{array}{cc}
                     I_{(d_D-1)m} &  \\
                      & -D_{0} \\
                   \end{array}
                 \right], \nonumber \\
                 \label{0nnfq}
N_{i} &:=& \left[
  \begin{array}{cccc}
    I_{(d_D-i-1)m} &  & & \\
     & -D_{i} & I_{p} & \\
     & I_{m} & 0  &  \\
     &   &   & I_{(i-1)m}\\
  \end{array}
\right], \,\, i= 1, \ldots, d_D-1.
\end{eqnarray}

Let $d=\max\{d_A,d_D\}$ and $r=\min\{d_A,d_D\}$ and introduce
$(d_A n+d_D m) \times (d_A n+d_D m)$
matrices $\mathbb{M}_0, \ldots, \mathbb{M}_{d}$ via
\begin{eqnarray}\nonumber
\mathbb{M}_0 &=& \left[
                        \begin{array}{c|c}
                          M_0 & (e_{d_A} e_{d_D}^{T}) \otimes  B\\
                          \hline
                          -(e_{d_D} e_{d_A}^{T})\otimes C & N_0 \\
                        \end{array}
                      \right] \\ \nonumber
\mathbb{M}_{d} &:=& \left[
                        \begin{array}{c|c}
                          M_{d_A} &  \\
                          \hline
                           & N_{d_D} \\
                        \end{array}
                      \right], \\
\mathbb{M}_i &:=& \left[
              \begin{array}{c|c}
                M_i &  \\
                \hline
                 & N_i \\
              \end{array}
            \right],\ i=1,\ldots,r-1,\label{fmbeh}
\end{eqnarray}
and
\begin{eqnarray}\nonumber
\mathbb{M}_i &:=& \left[
                   \begin{array}{cccc|c}
                     I_{(d_A-i-1)n} &  &  &  &  \\
                      & -A_i & I_n &  &  \\
                      & I_n & 0 &  &  \\
                      &  &  & I_{(i-1)n} &  \\
                     \hline
                      &  &  &  & I_{d_D m} \\
                   \end{array}
                 \right] \\ 
                 &=& \left[
              \begin{array}{c|c}
                M_i &  \\
                \hline
                 & I_{d_D m} \\
              \end{array}
            \right],\ i= r, r+1,\ldots,d_A-1, \ \mbox{\rm if }\; d_D < d_A,\label{fmbeh1}\\
\mathbb{M}_i &:=& \left[
                   \begin{array}{c|cccc}                     I_{d_A n} &  &&&\\
                           \hline
         & I_{(d_D-i-1)m} &  &  &  \\
         &  & -D_i & I_p &  \\
         &  & I_m & 0 &  \\
         &  &  &  & I_{(i-1)m} \\
    \end{array}
  \right] \nonumber \\
                 &=& \left[
              \begin{array}{c|c}
                I_{d_A n} &  \\
                \hline
                 & N_i\\
              \end{array}
            \right],\ i=d_A, d_A+1,\ldots,d_D-1, \ \mbox{if}\; d_D> d_A.\label{fmbeh2}
\end{eqnarray}
Observe that $\mathbb{M}_i \mathbb{M}_j = \mathbb{M}_j \mathbb{M}_i$  for  $ |i-j| > 1$ and all $\mathbb{M}_i$ (except possibly $\mathbb{M}_0$, $\mathbb M_{d}$) are invertible. Using these Fiedler matrices $\mathbb M_i$, then in \cite{BehB22} Fiedler pencils for RSMPs were defined as follows.
\begin{definition}\label{def:fiedlerS}
Consider an RSMP  $S(\lambda)$ as in (\ref{smpq}) and let $d=\max\{d_A,d_D\}$, $r=\min\{d_A,d_D\}$.
Let $\mathbb{M}_{0}, \ldots, \mathbb{M}_{d}$ be Fiedler matrices associated with $S(\lam)$ as in \eqref{fmbeh}, \eqref{fmbeh1}, and \eqref{fmbeh2}. Given any bijective map $\sigma: \{0, 1, \ldots, d-1\} \rightarrow \{1, 2, \ldots, d\}$, the matrix pencil
\begin{equation}
\mathbb{L}_{\sigma}(\lambda) := \lambda \mathbb{M}_{d}- \mathbb{M}_{\sigma^{-1} (1)}\cdots \mathbb{M}_{\sigma^{-1} (d)} =: \lambda \mathbb{M}_{d}- \mathbb{M}_{\sigma} \label{fps}
\end{equation}
is  called {\em Fiedler pencil} of  $S(\lam)$, respectively that of  $R(\lam),$ where $\sigma(i)$ describes the position of the factor $\mathbb{M}_{i}$ in the product $\mathbb{M}_{\sigma}$; i.e., $\sigma(i) = j$ means that $\mathbb{M}_{i}$ is the $j$th factor in the product.
\end{definition}
An $n \times n$ matrix polynomial $A(\lam)$ is said to be unimodular if the determinant of $A(\lam)$ is a nonzero constant for all $\lam \in \C$. 

In the next section we extend this construction of Fiedler pencils to the rectangular case.

\section{Fiedler pencils for rectangular Rosenbrock system matrix polynomials}\label{sec:RFP}
In this section we extend the results from \cite{AlaB16,Beh14} on Fiedler pencils for transfer functions of the form (\ref{trfun}) to rectangular RSMPs. 

Let $D(\lam)=\sum_{j=0}^{d_D}\lambda^j D_j \in \C[\lam]^{p, m}$, $C\in \C^{p,  n}$, $B \in \C^{n, m}$, let $A(\lam)= \sum_{i=0}^{d_A} \lambda^{i}A_i\in {\mathbb C}[\lambda]^{n,n}$ be regular and consider the associated  rectangular RSMP  $S(\lam)$ and  the transfer function $R(\lam)$.

The most simple way to perform a direct linearization of  $S(\lam)$ is to consider a \emph{first companion form}  $\mathcal{C}_1(\lam) $ and \emph{second companion form} $\mathcal{C}_2(\lam)$ given in (\ref{firstcomp}) and (\ref{secondcomp}), respectively.

To illustrate the differences  with the square case when introducing Fiedler pencils for rectangular  transfer functions $R(\lambda)$ we consider some examples.
\begin{example}\label{ex3}
Let $R(\lam) = D(\lam) + CA(\lam)^{-1}B \in  \C(\lam)^{p, m},$ where $A(\lam) = A_0+\lam A_1 + \lam^2 A_2 + \lam^3 A_3\in  \C[\lam]^{n, n}$  and $D(\lam)=  \lam D_1+D_0 \in \C[\lam]^{p, m}$. Let $\sig_1 = (1, 3, 2)$ and $\sig_2 = (2, 3, 1)$ be bijections from $\{0, 1, 2\}$ to $\{1, 2, 3\}$. 
To form the Fiedler pencils $\mathbb{L}_{\sig_1}(\lam) = \lam \mathbb{M}_3 - \mathbb{M}_0 \mathbb{M}_2 \mathbb{M}_1$ and $\mathbb{L}_{\sig_2}(\lam) = \lam \mathbb{M}_3 - \mathbb{M}_2 \mathbb{M}_0 \mathbb{M}_1$,
if $R(\lam)$ is square, i.e., $p = m $, then we can proceed as in Definition~\ref{def:fiedlerS},  and by using the commutativity relation we obtain $\mathbb{L}_{\sig_1}(\lam)= \mathbb{L}_{\sig_2}(\lam)$.

However, if  $R(\lam)$ is  non-square i.e., if $p \neq m$, then in general we cannot use the same construction, since in this case the matrix products may not be well defined.
To overcome this difficulty,  for each bijection $\sig$, we define the Fiedler matrices in a different way so that the products in
$\mathbb{M}_{\sig}$ are always well defined.

To obtain that $\mathbb{L}_{\sig_1}(\lam) = \mathbb{L}_{\sig_2}(\lam) $, for $\mathbb{L}_{\sig_1}(\lam)=\lambda \mathbb {M}_3-\mathbb{M}_0\mathbb{M}_2 \mathbb{M}_1$ we may
use
\begin{eqnarray*}\mathbb{M}_0 &=& \left[
                   \begin{array}{ccc|c}
                    I_n & 0 & 0 & 0_{n \times m} \\
                     0 & I_n & 0 & 0_{n \times m} \\
                     0 & 0 & -A_0 & B \\
                     \hline
                     0_{p \times n} & 0_{p \times n} & -C & -D_0 \\
                   \end{array}
                 \right], \
\mathbb{M}_1 = \left[
                   \begin{array}{ccc|c}
                    I_n & 0 & 0 & 0_{n \times m} \\
                     0 & -A_1 & I_n & 0_{n \times m} \\
                     0 & I_n & 0 & 0_{n \times m} \\
                     \hline
                     0_{m \times n} & 0_{m \times n} & 0_{m \times n} & I_{m} \\
                   \end{array}
                 \right],\\
\mathbb{M}_2 &=& \left[
                   \begin{array}{ccc|c}
                    -A_2 & I_n & 0 & 0_{n \times m} \\
                     I_n & 0 & 0 & 0_{n \times m} \\
                     0 & 0 & I_n & 0_{n \times m} \\
                     \hline
                     0_{m \times n} & 0_{m \times n} & 0_{m \times n} & I_m \\
                   \end{array}
                 \right],
                 \end{eqnarray*}
and for $\mathbb{L}_{\sig_2}(\lam) = \lam \mathbb{M}_3 - \mathbb{M}_2 \mathbb{M}_0 \mathbb{M}_1$, we may use
\begin{eqnarray*}\mathbb{M}_0 &=& \left[
                   \begin{array}{ccc|c}
                    I_n & 0 & 0 & 0_{n \times m} \\
                     0 & I_n & 0 & 0_{n \times m} \\
                     0 & 0 & -A_0 & B \\
                     \hline
                     0_{p \times n} & 0_{p \times n} & -C & -D_0 \\
                   \end{array}
                 \right],\
\mathbb{M}_1 = \left[
                   \begin{array}{ccc|c}
                    I_n & 0 & 0 & 0_{n \times m} \\
                     0 & -A_1 & I_n & 0_{n \times m} \\
                     0 & I_n & 0 & 0_{n \times m} \\
                     \hline
                     0_{m \times n} & 0_{m \times n} & 0_{m \times n} & I_m \\
                   \end{array}
                 \right],\\
\mathbb{M}_2 &=& \left[
                   \begin{array}{ccc|c}
                    -A_2 & I_n & 0 & 0_{n \times p} \\
                     I_n & 0 & 0 & 0_{n \times p} \\
                     0 & 0 & I_n & 0_{n \times p} \\
                     \hline
                     0_{p \times n} & 0_{p \times n} & 0_{p \times n} & I_p \\
                   \end{array}
                 \right]. 
\end{eqnarray*}
\end{example}

\begin{example}\label{ex2}
Let $R(\lam) = D(\lam) + CA(\lam)^{-1}B \in \C(\lam)^{p, m}$ with $A(\lam) = A_0+\lam A_1 + \lam^2 A_2 + \lam^3 A_3\in \C[\lam]^{n, n}$ and let $D(\lam) = D_0+\lam D_1 + \lam^2 D_2 \in \C[\lam]^{p, m}$. Let $\sig_1 = (1, 3, 2)$ and $\sig_2 = (2, 3, 1)$. 
Then $\mathbb{L}_{\sig_1}(\lam) = \lam \mathbb{M}_3 - \mathbb{M}_0 \mathbb{M}_2 \mathbb{M}_1$ and $\mathbb{L}_{\sig_2}(\lam) = \lam \mathbb{M}_3 - \mathbb{M}_2 \mathbb{M}_0 \mathbb{M}_1$.

If $R(\lam)$ is rectangular, then, as in Example~\ref{ex1},  we have to construct the Fiedler matrices in such a way that the matrix products $\mathbb{M}_0 \mathbb{M}_2 \mathbb{M}_1$ and $\mathbb{M}_2 \mathbb{M}_0 \mathbb{M}_1$ are defined.
We can e.g. use

\begin{eqnarray*}
\mathbb{M}_0 &=& \left[
                   \begin{array}{ccc|cc}
                    I_n & 0 & 0 & 0 & 0_{n \times m} \\
                     0 & I_n & 0 & 0 & 0_{n \times m} \\
                     0 & 0 & -A_0 & 0 &  B \\
                     \hline
                      0 & 0 & 0 & I_m &  0 \\
                     0_{p \times n} & 0_{p \times n} & -C & 0 & -D_0 \\
                   \end{array}
                 \right],\\
\mathbb{M}_1 &=& \left[
                   \begin{array}{ccc|cc}
                    I_n & 0 & 0 & 0_{n \times m} & 0 \\
                     0 & -A_1 & I_n & 0_{n \times m} & 0 \\
                     0 & I_n & 0 & 0_{n \times m} & 0 \\
                     \hline
                      0 & 0 & 0 & -D_1 &  I_p \\
                     0_{m \times n} & 0_{m \times n} & 0_{m \times n} & I_m & 0 \\
                   \end{array}
                 \right],\\
\mathbb{M}_2 &=& \left[
                   \begin{array}{ccc|cc}
                    -A_2 & I_n & 0 & 0 & 0_{n \times m} \\
                     I_n & 0 & 0 & 0 & 0_{n \times m} \\
                     0 & 0 & I_n & 0 & 0_{n \times m} \\
                     \hline
                       0 & 0 & 0 & I_m &  0 \\
                     0_{m \times n} & 0_{m \times n} & 0_{m \times n}& 0 & I_m \\
                   \end{array}
                 \right]. 
\end{eqnarray*}
Then $\mathbb{M}_0\mathbb{M}_2 \mathbb{M}_1$ is well defined. Similarly, for $\mathbb{L}_{\sig_2}(\lam) = \lam \mathbb{M}_3 - \mathbb{M}_2 \mathbb{M}_0 \mathbb{M}_1$ we may use
\begin{eqnarray*}\mathbb{M}_0 &=& \left[
                   \begin{array}{ccc|cc}
                    I_n & 0 & 0 & 0 & 0_{n \times m} \\
                     0 & I_n & 0 & 0 & 0_{n \times m} \\
                     0 & 0 & -A_0 & 0 &  B \\
                     \hline
                      0 & 0 & 0 & I_p &  0 \\
                     0_{p \times n} & 0_{p \times n} & -C & 0 & -D_0 \\
                   \end{array}
                 \right],  \\
\mathbb{M}_1 &=& \left[
                   \begin{array}{ccc|cc}
                    I_n & 0 & 0 & 0_{n \times m} & 0 \\
                     0 & -A_1 & I_n & 0_{n \times m} & 0 \\
                     0 & I_n & 0 & 0_{n \times m} & 0 \\
                     \hline
                      0 & 0 & 0 & -D_1 &  I_p \\
                     0_{m \times n} & 0_{m \times n} &0_{m \times n} & I_m & 0 \\
                   \end{array}
                 \right],\\
                 \mathbb{M}_2 &=& \left[
                   \begin{array}{ccc|cc}
                    -A_2 & I_n & 0 & 0 & 0_{n \times p} \\
                     I_n & 0 & 0 & 0 & 0_{n \times p} \\
                     0 & 0 & I_n & 0 & 0_{n \times p}\\
                     \hline
                       0 & 0 & 0 & I_p &  0 \\
                     0_{p \times n} & 0_{p \times n} & 0_{p \times n}& 0 & I_p \\
                   \end{array}
                 \right].
\end{eqnarray*}
Note that the sizes of $\mathbb{M}_2$ and $\mathbb{M}_0$ in $\mathbb{L}_{\sig_1}(\lam)$ are different from the sizes of the corresponding $\mathbb{M}_2$ and $\mathbb{M}_0$ in $\mathbb{L}_{\sig_2}(\lam)$. 
\end{example}

These examples indicate that to define Fiedler pencils for a rectangular RSMP $S(\lam)$, the sizes of the matrices $\mathbb{M}_i$ have to depend on the specific choice of the bijection $\sig$. To do this, we use of the following definition from \cite{DeTDM09}.  We use the notation $(t : p)$ for ordered tuple of indices consisting of consecutive integers from $t$ to $p$.
\begin{definition} \label{def:consinv}
Let $\sig: \{0, 1, \ldots, k-1\} \rightarrow \{1, \ldots, k\}$ be a bijection.
\begin{itemize}
\item[(a)] For $i = 0, \ldots, k-2$, we say that $\sig$ has a consecution at $i$ if $\sig(i) < \sig(i+1)$, and that $\sig$ has an inversion at $i$ if $\sig(i) > \sig(i+1)$

\item[(b)] Denote by  $\mathfrak{c}(\sig)$ the total number of consecutions in $\sig$, and by $\mathfrak{i}(\sig)$ the total number of inversions in $\sig$.

\item[(c)] For $i \leq j$, we denote by  $\mathfrak{c}(\sig(i:j))$ the total number of consecutions that $\sig$ has at $i, i+1, \ldots, j$, and by 
$\mathfrak{i}(\sig(i:j))$ the total number of inversions that $\sig$ has at $i, i+1, \ldots, j. $ Observe that  $\mathfrak{c}(\sig) = \mathfrak{c}(\sig(0:k-2))$ and $\mathfrak{i}(\sig) = \mathfrak{i}(\sig(0:k-2))$.

\item[(d)] The consecution-inversion structure sequence of $\sig$, denoted by CISS$(\sig)$, is the tuple $(c_1, i_1, c_2, i_2, \ldots, c_l, i_l),$ where $\sig$ has $c_1$ consecutive consecutions at $0, 1, \ldots, c_1-1;$ $i_1$ consecutive inversions at $c_1, c_1+1, \ldots, c_1+i_1-1$ and so on, up to $i_l$ inversions at $k-1-i_l, \ldots, k-2. $
\end{itemize}
\end{definition}

Then we introduce the following algorithm (using {MATLAB} notation) for the construction  of Fiedler pencils for  general RSMPs for the case $d_A> d_D$. The case $d_A = d_D$ is treated in Remark~\ref{rem:equaldegree}.
\begin{algorithm}[H]
\caption{Construction of $\mathbb{M}_{\sigma}$ for $\mathbb{L}_{\sigma}(\lam) := \lam \mathbb{M}_{d_{A}}- \mathbb{M}_{\sigma}$.}
\label{alg1}

\textbf{Input}: $ S(\lam) = \left[
                              \begin{array}{c|c}
                                \sum \limits_{i=0}^{d_A}\lam^{i}A_{i} & -B \\
                                \hline
                                C & \sum \limits_{i=0}^{d_D}\lam^{i}D_{i}\\
                              \end{array}
                            \right]$ and a bijection $\sigma :\{0, 1, \ldots, d_{A}-1\} \rar \{1, 2, \ldots, d_{A}\}$. \\
\textbf{Output}: { $\mathbb{M}_{\sigma}$ }
\begin{algorithmic}

\If{$\sigma$ has a consecution at $0$}
    \State $\mathbb{W}_0 := \left[
                    \begin{array}{cc|cc}
                      -A_{1} & I_n & 0 & 0 \\
                      -A_0 & 0 & B & 0 \\
                      \hline
                      0 & 0 & -D_1 & I_p \\
                      -C & 0 & -D_0 & 0 \\
                    \end{array}
                  \right]$

\Else
    \State $\mathbb{W}_0 := \left[
                    \begin{array}{cc|cc}
                      -A_{1} & -A_0 & 0 & B \\
                       I_n & 0 & 0 & 0 \\
                      \hline
                      0 & -C & -D_1 & -D_0 \\
                      0 & 0 & I_m & 0 \\
                    \end{array}
                  \right]$

\EndIf

\For{$i = 1:d_{D}-2$}

\If{$\sigma$ has a consecution at $i$}
 \State {\scriptsize $\mathbb{W}_i :=  \left[
	\begin{array}{ccc|c}
		-A_{i+1} & I_{n} &  0  & \\
		\mathbb{W}_{i-1}(1:i+1,1) & 0 & \mathbb{W}_{i-1}(1:i+1,2:i+1) & W_{12}  \\
		\hline
		0 & 0 &  0 & \\
		\mathbb{W}_{i-1}(2+i:2i+2, 1) & 0 & \mathbb{W}_{i-1}(2+i:2i+2,2:i+1)    & W_{22}  \\  
	\end{array}
	\right]$},   where  \\
\qquad\qquad{\scriptsize$W_{12}= \left[
	\begin{array}{ccc}
		0 & 0 &  0 \\
		\mathbb{W}_{i-1}(1:i+1, i+2) & 0 & \mathbb{W}_{i-1}(1:i+1, i+3:2i+2)  \\
	\end{array}
	\right]$}, \\ 
\qquad\qquad{\scriptsize $W_{22}= \left[
	\begin{array}{ccc}
		-D_{i+1} & I_{p} &  0     \\
		\mathbb{W}_{i-1}(2+i:2i+2,i+2) & 0 & \mathbb{W}_{i-1}(2+i:i2+2, i+3:2i+2)  \\
	\end{array}
	\right]$.} 
\Else
    \State {\scriptsize$\mathbb{W}_i := \left[
           \begin{array}{cc|cc}
             -A_{i+1} & \mathbb{W}_{i-1}(1,1:i+1) & 0 & \mathbb{W}_{i-1}(1,2+i:2i+2) \\
             I_{n} & 0  & 0 &  0 \\
             0 & \mathbb{W}_{i-1}(2:i+1,1:i+1) & 0 & \mathbb{W}_{i-1}(2:i+1,2+i:2i+2) \\
             \hline
             0 & \mathbb{W}_{i-1}(i+2,1:i+1) & -D_{i+1} & \mathbb{W}_{i-1}(i+2,2+i:2i+2) \\
             0 &  0 & I_{m} &   0    \\
             0 & \mathbb{W}_{i-1}(i+3:2i+2,1:i+1) & 0  & \mathbb{W}_{i-1}( i+3:2i+2,2+i:2i+2)  \\
           \end{array}
         \right] $}
\EndIf

\EndFor

\For{$i = d_{D}-1: d_{A}-2$}

\If{$\sigma$ has a consecution at $i$}
    \State $\mathbb{W}_i :=  \left[
       \begin{array}{cccc}
         -A_{i+1} & I_{n} & 0 & 0  \\
         \mathbb{W}_{i-1}(:,1) & 0 & \mathbb{W}_{i-1}(:,2:i+1) & \mathbb{W}_{i-1}(:,i+2:d_{D}+i+1)  \\
       \end{array}
     \right]$

\Else
    \State $\mathbb{W}_i := \left[
           \begin{array}{cc}
             -A_{i+1} & \mathbb{W}_{i-1}(1,:)   \\
             I_{n} & 0    \\
             0 & \mathbb{W}_{i-1}(2:i+1, :)  \\
             0 & \mathbb{W}_{i-1}(i+2:d_{D}+i+1, :)   \\
           \end{array}
         \right] $
\EndIf

\EndFor

\State $\mathbb{M}_{\sigma} := \mathbb{W}_{d_{A}-2}$

\end{algorithmic}
\end{algorithm}

The construction in  Algorithm~1 
leads to the following properties.
\begin{lemma}\label{cowi}
Let  $S(\lam)$ be as in (\ref{rosmatrix}) with $A(\lam) = \sum\limits_{i=0}^{d_A}\lam^{i}A_i\in \C[\lam]^{n, n}$ and $D(\lam) = \sum \limits_{i=0}^{d_D}\lam^{i}D_{i} \in \C[\lam]^{p, m}$ with $d_A> d_D$.  Let $\sig$ be a bijection and let $\mathbb{W}_0, \mathbb{W}_1, \ldots, \mathbb{W}_{d_{A}-2} $  be the matrices provided by Algorithm~\ref{alg1}.
\begin{itemize}
\item[(a)] For $i = 0:d_{D}-2$, the size of $\mathbb{W}_i$  is 
\begin{small}
\begin{eqnarray*}& &\left[\left(n+n \mathfrak{c}(\sig(0:i)) + n \mathfrak{i}(\sig(0:i))\right) + \left(p+p \mathfrak{c}(\sig(0:i)) + m \mathfrak{i}(\sig(0:i))\right)\right] \times \\
&& \left[\left(n + n \mathfrak{c}(\sig(0:i)) + n \mathfrak{i}(\sig(0:i))\right) + \left(m + p \mathfrak{c}(\sig(0:i)) + m \mathfrak{i}(\sig(0:i))\right)\right] 
\end{eqnarray*}
\end{small}
and for $i = d_{D}-1: d_{A}-2$  the size of $\mathbb{W}_i$ is 
\begin{small}
\begin{eqnarray*}
&&\left[\left(n+n \mathfrak{c}(\sig(0:i)) + n \mathfrak{i}(\sig(0:i))\right) + \left(p+p \mathfrak{c}(\sig(0:d_D-2)) + m \mathfrak{i}(\sig(0:d_D-2))\right)\right]  \times \\
&&
\left[\left(n+n \mathfrak{c}(\sig(0:i)) + n \mathfrak{i}(\sig(0:i))\right) + \left(m + p \mathfrak{c}(\sig(0:d_D-2)) + m \mathfrak{i}(\sig(0:d_D-2))\right)\right].
\end{eqnarray*}
\end{small}
\item[(b)] The $(1,1)$ block of $\mathbb{W}_{i}$ is $-A_{i+1}\in\mathbb{C}^{n\times n}$ and for $i=0:d_{D}-2$ the $(3+i,3+i)$ block of $\mathbb{W}_{i}$ is $-D_{i+1}\in\mathbb{C}^{p\times m}$ and for $i=d_{D}-1:d_{A}-2$, the $(3+i,3+i)$ block of $\mathbb{W}_{i}$ is $-D_{d_{D}-1}\in\mathbb{C}^{p\times m}$. The remaining diagonal blocks of $\mathbb{W}_{i}$ are square zero matrices, and more precisely, $\mathbb{W}_i(i+2-j, i+2-j) = 0_n$ for $j = 0, 1, \ldots, i$ and $i= 0, 1, \ldots d_{{A}}-2$,
 and 
\[\mathbb{W}_i(2i+4-j, 2i+4-j) =
\begin{cases}
0_p, & \text{if } \sig \text{ has a consecution at } j \\
0_m, &  \text{if } \sig \text{ has an inversion at } j
\end{cases}
\]
 for $j = 0, 1, \ldots, i$ and $i= 0, 1, \ldots d_{{D}}-2$ and
    $$\mathbb{W}_{i}(d_{D}+2-j,d_{D}+2-j)=\begin{cases}
                                     0_p, \,\,\,\,\ &\text{if $\sigma$ has a conseqution at $j$} \\
                                     0_m, &\text{if $\sigma$ has an inversion at $j$}
                                     \end{cases}$$
    for $j=0:d_{D}-2$ and $i=d_{D}-1:d_{A}-2$.

\item[(c)] If $\sig$ has a consecution at $i$, then the size of the zero block in the $W_{21}$ block of $\mathbb{W}_i$ is $0_{p \times n}$ and if $\sig$ has an inversion at $i$, then the size of the zero block in the $W_{12}$ block of $\mathbb{W}_i$  is $0_{n \times m}$.
\end{itemize}
\end{lemma}
\begin{proof}
The proof follows that of  [\cite{DeDM12a}, Th. 3.6]. Note that $\mathbb{W}_0$ is well-defined. Therefore $\mathbb{W}_1$ is also well-defined for either $\sig$ having a consecution at $1$ or an inversion at $1$. This can be seen as follows. Suppose that $\sig$ has a consecution at $1$, then
\[
\mathbb{W}_1 = \left[
          \begin{array}{cccccc}
            -A_2 & I_n & 0 & 0  & 0  & 0 \\
            \mathbb{W}_0(1:2, 1) & 0 & \mathbb{W}_{0}(1:2, 2) & \mathbb{W}_0(1:2, 3)& 0 & \mathbb{W}_0(1:2, 4) \\
             0 & 0 & 0 & -D_2  & I_p  & 0 \\
            \mathbb{W}_0(3:4, 1) & 0 & \mathbb{W}_{0}(3:4, 2) & \mathbb{W}_0(3:4, 3)& 0 & \mathbb{W}_0(3:3, 4) \\
          \end{array}
        \right]
\]
and if $\sig$ has an inversion at $1$ then
\[
\mathbb{W}_1 = \left[
    \begin{array}{cccc}
      -A_2 & \mathbb{W}_0(1, 1:2) & 0 & \mathbb{W}_0(1, 3:4) \\
      I_n & 0 & 0 & 0 \\
      0 & \mathbb{W}_0(2, 1:2)& 0 & \mathbb{W}_0(2, 3:4) \\
      0 & \mathbb{W}_0(3, 1:2) & -D_2 &  \mathbb{W}_0(3, 3:4)  \\
      0 & 0 & I_m & 0 \\
      0 & \mathbb{W}_0(4, 1:2) & 0 & \mathbb{W}_0(4, 3:4) \\
    \end{array}
  \right].
\]
In both cases $\mathbb{W}_1(1,1) = -A_2$ and $\mathbb{W}_1(4,4) = -D_2$.  Also, $\mathbb{W}_0(:, 1)$ has $n$ columns and $\mathbb{W}_0(1, :)$ has $n$ rows and $\mathbb{W}_0(:, 3)$ has $m$ columns and $\mathbb{W}_0(3, :)$ has $p$ rows. Hence in both cases $\mathbb{W}_1$ is well-defined. The same argument can be applied inductively to show that $\mathbb{W}_2, \ldots, \mathbb{W}_{d_{A}-2}$ are  well-defined. Note that, by definition of $\mathbb{W}_0, $ it follows that $\mathbb{W}_i$ is partitioned into $(2i+4) \times (2i+4)$ blocks and for the remaining matrices up to index $d_D$ in the sequence in each step of the "for" loop of Algorithm~\ref{alg1}, two block rows and two block columns are added  and,  for the indices $d_{D+1}$ to $d_A$, one block row and one block column is added in each step.

\begin{itemize}
\item[(a)] Note that the size of $\mathbb{W}_0$ is $(2n+2p)\times (2n+m+p)$ if $\sig$ has a consecution at $0$ and $(2n+p+m)\times (2n+2m)$ if $\sig$ has an inversion at $0$. Hence part $(a)$ holds for $\mathbb{W}_0$ and for the remaining matrices the result is as follows: Up to index $d_{D}-2$ we have the following:
 
If $\sig$ has a consecution at $i$ then $\mathbb{W}_i$ has $n+p$ rows and $n+m$ more columns  than $\mathbb{W}_{i-1}$, and if $\sig$ has an inversion at $i$ then $\mathbb{W}_i$ has $n+p$ rows and $n+m$ more columns  than $\mathbb{W}_{i-1}$.

For indices $d_{D}-1: d_{A}-2$ we have the following:

If $\sig$ has a consecution at $i$ then $\mathbb{W}_i$ has $n$ rows and $n$  more columns than $\mathbb{W}_{i-1}$ and if $\sig$ has an inversion at $i$ then $\mathbb{W}_i$ has $n$ rows and $n$  more columns than $\mathbb{W}_{i-1}$.  Hence $(a)$ is proved.

\item[(b)]
From the construction of $\mathbb{W}_{i}$ it is clear that for $i=0:d_{A}-2$, the $(1,1)$ diagonal block of $\mathbb{W}_{i}$ is $-A_{i+1}$ and for $i=0:d_{D}-2$, the $(3+i,3+i)$ block of $\mathbb{W}_{i}$ is is $-D_{i+1}$ and for $i=d_{D}-1:d_{A}-1$, the $(3+i,3+i)$ block of $\mathbb{W}_{i}$ is $-D_{d_{D}-1}$. The remaining diagonal blocks of  $\mathbb{W}_i$  are square zero matrices $0_n$ for blocks $2,\ldots,(3+i)-1$, while for blocks from $3+i+1$ on, the diagonal block is either  $0_p$ or $0_m$ depending upon the consecutions and inversions at  $0, 1, \ldots, i$. This finishes the proof of $(b)$.

\item[(c)] The proof directly follows from the construction of $\mathbb{W}_i$. 
\end{itemize}
\end{proof}
A similar result for the case  $d_A< d_D$ is presented in the Appendix.

\begin{remark}\label{rem:equaldegree}{\rm 
Let  $S(\lam)$ be an $(n+p) \times (n+m)$ RSMP as in (\ref{rosmatrix}) with $A(\lam) = \sum\limits_{i=0}^{d_A}\lam^{i}A_i\in \C[\lam]^{n, n}, D(\lam) = \sum \limits_{i=0}^{d_D}\lam^{i}D_{i} \in \C[\lam]^{p, m}$ with $d_A= d_D$. For this case Algorithm~\ref{alg1} holds and will stop at $i= 0:d_{D}-2$ and the same Lemma~\ref{cowi} also holds up to $i= 0:d_{D}-2$. } 
\end{remark}

\begin{remark}\label{rem:remdiff}{\rm 
Consider Algorithm~\ref{alg1} in \cite{BehB22} for square RSMPs. Then Algorithm~\ref{alg1} differs only in the sizes of identity block matrices and zero block matrices added at each step of the construction which are chosen to fit the size $p \times m$ of the coefficients of the matrix polynomial $D(\lam)$ and to the matrices $ C$ and $B$. }
\end{remark}

Using the construction in Algorithm~\ref{alg1}, we have the following definition of Fiedler pencils for  RSMPs $S(\lambda) \in{\mathbb C}[\lambda]^{n+p,n+m}$.
\begin{definition}\label{def:FPRSMP} 
Let $S(\lam)$ be an $(n+p) \times (n+m)$ RSMP as in (\ref{rosmatrix}) with $D(\lam) = \sum \limits_{i=0}^{d_D}\lam^{i}D_{i} \in \C[\lam]^{p, m}$ and regular $A(\lam) = \sum\limits_{i=0}^{d_A}\lam^{i}A_i\in \C[\lam]^{n, n}$, where $d_A \geq d_D$. Let $\sig : \{0, 1, \ldots, d_{A}-1\} \rightarrow \{1, 2, \ldots, d_{A}\}$ be a bijection and denote by $\mathbb{M}_{\sig}$ the last matrix of the sequence constructed by Algorithm~\ref{alg1}, that is $\mathbb{M}_{\sig} = \mathbb{W}_{d_{A}-2}$. Then the pencil
\[
\mathbb{L}_{\sig}(\lam) = \lam \left[
                                   \begin{array}{cc|cc}
                                     A_{d_{A}} &  & & \\
                                      & I_{(d_{A}-1)n} & 0 &  \\
                                     \hline
                                      &  &  D_{d_{D}} & \\
                                      & 0 &  & I_{p \mathfrak{c}(\sig)+ m \mathfrak{i}(\sig)} \\
                                   \end{array}
                                 \right] - \mathbb{M}_{\sig}
\]
of size $(p + p \mathfrak{c}(\sig) + m \mathfrak{i}(\sig)) +d_{A}n \times (m + p \mathfrak{c}(\sig) + m \mathfrak{i}(\sig)) +d_{A}n$ is called \emph{Fiedler pencil} of $S(\lam)$ associated with $\sig$.
\end{definition}
\begin{remark}{\rm}
\begin{itemize}
\item[(1)] When $p  \neq  m $, the size of the zero matrices in the blocks $(1,2)$  and $(2, 1)$ in the leading coefficient of the Fiedler pencil $\mathbb{L}_{\sig}(\lam)$  defined in  
Definition~\ref{def:FPRSMP} are different.
Furthermore, there are Fiedler pencils associated with  $S(\lam)$ with several different sizes, because for $d = \max(d_{A}, d_{D})$ the sum $\mathfrak{c}(\sig) + \mathfrak{i}(\sig) = d-1$ is fixed for all $\sig$ and so different pairs $(\mathfrak{c}(\sig), \mathfrak{i}(\sig))$ produce different sizes of $\mathbb{L}_{\sig}(\lam). $
\item[(2)] In the definition of Fiedler pencils for rectangular  $S(\lam)$,  the bijection is not really needed, rather we 
need the sequence of decisions that we have identified with $\sig$ having a consecution or inversion.
\end{itemize}
\end{remark}

\begin{example}
Suppose that  $\sig = (1, 2, 4, 3, 6, 5) $ and $D(\lam) = \lam D_1+D_0$. Then $\mathbb{M}_{\sig} = \mathbb{M}_0 \mathbb{M}_1 \mathbb{M}_3 \mathbb{M}_2 \mathbb{M}_5 \mathbb{M}_4=\mathbb{M}_0 \mathbb{M}_1 \mathbb{M}_3 \mathbb{M}_5 \mathbb{M}_2 \mathbb{M}_4 $, since $\mathbb M_2$ and $\mathbb M_5$ commute. If $S(\lam)$ is rectangular and $d_{A} = 6$, then 
\[
\mathbb{M}_{\sig}=\mathbb{W}_4 = \left[
              \begin{array}{cccccc|c}
                -A_5 & -A_4 & I_m & 0 & 0 & 0 & 0 \\
                I_n & 0 & 0 & 0 & 0 & 0 & 0 \\
                0 & -A_3 & 0 & -A_2 & I_m & 0 & 0 \\
                0 & I_n & 0 & 0 & 0 & 0 & 0 \\
                0 & 0 & 0 & -A_1 & 0 & I_m & 0 \\
                0 & 0 & 0 & -A_0 & 0 & 0 & B \\
                \hline
                0 & 0 & 0 & -C & 0 & 0 & -D_0 \\
              \end{array}
            \right].
\]
If we consider a square  $S(\lam)$ with the same $\sig$, then a direct multiplication of Fiedler matrices gives
\[
\mathbb{W}_4 = \left[
              \begin{array}{cccccc|c}
                -A_5 & -A_4 & I_n & 0 & 0 & 0 & 0 \\
                I_n & 0 & 0 & 0 & 0 & 0 & 0 \\
                0 & -A_3 & 0 & -A_2 & I_n & 0 & 0 \\
                0 & I_n & 0 & 0 & 0 & 0 & 0 \\
                0 & 0 & 0 & -A_1 & 0 & I_n & 0 \\
                0 & 0 & 0 & -A_0 & 0 & 0 & B \\
                \hline
                0 & 0 & 0 & -C & 0 & 0 & -D_0 \\
              \end{array}
            \right].
\]
Observe that the matrices $\mathbb W_4$ look the same in both cases but the sizes of zero blocks in the last block row and  columns are not the same. Also, the sizes of some identity blocks are different as per the size of the matrix $A_i$.  
\end{example}
In the next section we analyze the properties of the constructed Fiedler pencils.

\section{Fiedler pencils of rectangular  $S(\lam)$ are linearizations} \label{sec:LFP}
In this section we show that the Fiedler pencils that we have constructed for RSMPs are indeed linearizations. For this we have to recall a few basic results.
\begin{definition}\label{def:syseq} Let $ S_1(\lam)$ and $S_2(\lam)$ be $(n+p) \times (n+m)$ RSMPs of the form (\ref{rosmatrix}) that are partitioned conformably.  Then $ S_1(\lambda)$ is said to be \emph{system equivalent} to $S_2(\lambda)$ (denoted as $ S_1(\lam) \thicksim_{se}  S_2(\lam)$), if there exist unimodular matrix polynomials $ U(\lambda), V(\lambda)\in \mathbb F^{n,n}$, $ \widetilde{U}(\lam)\in \mathbb F^{p,p}$, and $\widetilde{V}(\lam)\in \mathbb F^{m,m}$ such that  for all   $\lam \in \C$ we have
\begin{equation}\label{sysequi}
 \left[\begin{array}{c|c} U(\lam) &  0 \\  \hline 0 & \widetilde{U}(\lam) \end{array} \right]  S_1(\lam)  \left[ \begin{array}{c|c} V(\lam) & 0 \\ \hline  0&  \widetilde{V}(\lam) \end{array}\right] = S_2 (\lam).
\end{equation}
\end{definition}
\begin{definition} \label{def:roslin} Let $ S(\lam)$ be an $ (n+p) \times (n+m)$ RSMP of the form (\ref{rosmatrix})
with degree $d=\max\{d_A,d_D\}$. A linear matrix polynomial $\mathbb{L}(\lambda)$ is called a \emph{Rosenbrock linearization} of $S(\lambda)$
if it has the form
\[
{\mathbb L}(\lambda):=
\left[
                                                     \begin{array}{c|c}
                                                       {\mathcal A}(\lam) & {\mathcal B}\\
                                                       \hline
                                                       {\mathcal C} & {\mathcal D}(\lam) \\
                                                     \end{array}
                                                   \right],
\]
with matrix polynomials ${\mathcal A}(\lambda),{\mathcal D}(\lambda)$ of degree less than or equal to $1$, constant 
matrices ${\mathcal B}, {\mathcal C}$,  and ${\mathbb L}(\lambda)$ is system equivalent to 
\begin{equation}\label{sysL}
\mathcal{\Tilde{S}}(\lambda) := \left[\begin{array}{c|c} U(\lam) &  0 \\  \hline 0 & \widetilde{U}(\lam) \end{array} \right] \mathbb{L}(\lam)  \left[ \begin{array}{c|c} V(\lam) & 0 \\ \hline  0&  \widetilde{V}(\lam) \end{array}\right] =\left[
                                           \begin{array}{c|c}
                                             I_{(d-1)n} & 0 \\
                                             \hline
                                             0 &   S(\lam) \\
                                           \end{array}
                                         \right],
\end{equation}
where  $U(\lam)$, $V(\lam)$, $\tilde{U}(\lam)$, and $\tilde{V}(\lambda)$ are unimodular matrix polynomials.
If, in addition, $U(\lam)$, $V(\lam)$, $\tilde{U}(\lam)$, and $\tilde{V}(\lam)$ in  (\ref{sysL}) are constant matrices, then $\mathbb{L}(\lam)$ is said to be a \emph{strict Rosenbrock linearization} of  $S(\lam)$.
\end{definition}
For  RSMPs  the first and second companion forms $\mathcal C_1(\lam)$ and $\mathcal C_2(\lam)$ are Rosenbrock linearizations but their sizes are different. Also, note that in Definition~\ref{def:roslin} we are allowed to have any output dimension $p \geq 0$, whereas if  $S(\lam)$ is square of size $(n +r) \times (n +r)$ then the size of $p$ in $\mathbb{L}(\lam)$ is fixed and of size $(d_{A}-1)n$, 
see \cite{AlaB16}.

Another concept that needs to be extended to rectangular Rosenbrock matrix polynomials is that of 
Horner shifts, see \cite{DeTDM09}.
Let $P(\lam) = A_0 + \lam A_1 + \cdots + \lambda^{d_A}A_{d_A} \in {\mathbb F}[\lambda]^{p,m}$. 
For $k = 0, \ldots, d_A$, the degree $k$ \emph{Horner shift} of $P(\lambda)$ is the matrix polynomial $ P_{k}(\lambda) := A_{d_A-k} + \lambda A_{d_A-k+1} + \cdots + \lambda^{k}A_{d_A}$ and these satisfy
\begin{align}
& P_{0}(\lambda) = A_{d_A}, \nonumber \\
&  P_{k+1}(\lambda) = \lambda  P_{k}(\lambda) + A_{d_A-k-1} , \mbox{  for } 0\leq k\leq d_A-1, \nonumber\\
&  P_{d_A}(\lambda) =  P(\lambda).  \label{hs}
\end{align}

Assuming again that $d_{A} > d_{D}$, we have the following two algorithms (in MATLAB notation) which will be used to show that Fiedler pencils are linearizations of general RSMPs.

\begin{algorithm}[H]
\caption{The following algorithm constructs a sequence of matrix polynomials $\{\mathbb{N}_0, \mathbb{N}_1, \ldots, \mathbb{N}_{d_{A}-2}\}, $ where each matrix $\mathbb{N}_i$, for $i = 1,2, \ldots, d_{A}-2, $ is partitioned into blocks in such a way that the blocks of $\mathbb{N}_{i-1}$ are blocks of $\mathbb{N}_i$. Let $P_{k}(\lam)$ and $Q_{k}(\lambda)$ denote the degree $k$ Horner shifts of $A(\lam) $ and $D(\lam)$, respectively.   For simplicity we have omitted the argument $\lam$ in $P_{d_{A}}(\lam)$, $Q_{d_{D}}(\lam)$ and in  $\{\mathbb{N}_i\}_{i=0}^{d_{A}-2}$.}
\label{alg3}

\begin{algorithmic}
\If {$\sig$ has a consecution at $0$, }
\State $\mathbb{N}_0 = \left[
          \begin{array}{cc|cc}
            I_n & 0 & & \\
            \lam I_n & I_n & & \\
            \hline
             &  & I_p & 0 \\
            &  & \lam I_p & I_p   \\
          \end{array}
        \right]
$
\Else
\State $\mathbb{N}_0 = \left[
          \begin{array}{cc|cc}
            0 & -I_n & & \\
           I_n & P_{d_{A}-1} & & \\
            \hline
             &  & 0 & -I_m  \\
            &  & I_p & Q_{d_{D}-1}   \\
          \end{array}
        \right]
$
\EndIf

\For {$i=1:d_{D}-2$}
\If{$\sigma$ has a consecution at $i$}
    \State $\mathbb{N}_i=\left[
                    \begin{array}{c|c}
                    	N_i & \\ \hline
                    	& N_i'
                    \end{array}
                     \right]$, \\
where \\

 $N_i = \left[ \begin{array}{cc} I_n & 0 \\
 	 \lambda \mathbb{N}_{i-1}(1:i+1,1) & \mathbb{N}_{i-1}(1:i+1,1:i+1) \end{array} \right]$, \\
   $N_i' = \left[ \begin{array}{cc} I_p & 0 \\          
   \lambda\mathbb{N}_{i-1}(2+i:2i+2,i+2) & \mathbb{N}_{i-1}(2+i:2i+2,2+i:2i+2) \end{array} \right]$ \\

\Else
\State $\mathbb{N}_i = \left[
                \begin{array}{c|c}
                	N_i & \\ \hline & N_i'
               \end{array}
           	 \right]$,  \\ where
           	 
{\scriptsize $N_i = \left[ \begin{array}{ccc} 0 & -I_n & 0 \\ \mathbb{N}_{i-1}(1:i+1,1) & \mathbb{N}_{i-1}(1:i+1,1)P_{d_A-i-1} & \mathbb{N}_{i-1}(1:i+1,2:i+1) \end{array} \right] $, \\
$N_i' = \left[ \begin{array}{ccc}  0 & -I_m & 0 \\  \mathbb{N}_{i-1}(2+i:2i+2,i+2) & \mathbb{N}_{i-1}(2+i:2i+2,i+2)Q_{d_D-i-1} & \mathbb{N}_{i-1}(2+i:2i+2,3+i:2i+2) \end{array} \right]$}

\EndIf

\EndFor
           	 
\For {$i = d_{D}-1:d_{A}-2$}
\If{$\sigma$ has a consecution at $i$}
    \State $\mathbb{N}_i = \left[
                    \begin{array}{cc}
                      I_n & 0 \\
                      \lam \mathbb{N}_{i-1}(:, 1) & \mathbb{N}_{i-1} \\
                    \end{array}
                  \right]
    $
\Else
\State $\mathbb{N}_i = \left[
         \begin{array}{ccc}
           0 & -I_n & 0 \\
           \mathbb{N}_{i-1}(:, 1) & \mathbb{N}_{i-1}(:, 1)P_{d_{A}-i-1} & \mathbb{N}_{i-1}(:, 2:d_{D}+i+1)  \\
         \end{array}
       \right]
$
\EndIf

\EndFor

\end{algorithmic}
\end{algorithm}

\begin{algorithm}[H]
\caption{The following algorithm constructs a sequence of matrix polynomials $\{\mathbb{H}_0, \mathbb{H}_1, \ldots, \mathbb{H}_{d_{A}-2}\}, $ where each matrix $\mathbb{H}_i$, for $i = 1,2, \ldots, d_{A}-2, $ is partitioned into blocks in such a way that the blocks of $\mathbb{H}_{i-1}$ are blocks of $\mathbb{H}_i$. Let $P_{k}(\lam)$ and $Q_{k}(\lambda)$ denote the degree $k$ Horner shifts of $A(\lam) $ and $D(\lam)$, respectively.}
	\label{alg4}
\begin{algorithmic}

\If{$\sigma$ has a consecution at $0$}
    \State $\mathbb{H}_0 := \left[
                    \begin{array}{cc|cc}
                      0 & I_n & & \\
                      -I_n & P_{d_{A}-1} & & \\
                      \hline
                        &  & 0 & I_m \\
                     & & -I_p & Q_{d_{D}-1}  \\
                    \end{array}
                  \right]
$

\Else
    \State $\mathbb{H}_0 := \left[
                    \begin{array}{cc|cc}
                      I_{n} & \lam I_n & & \\
                       0 & I_n &  & \\
                      \hline
                     & & I_{m} & \lam I_m   \\
                      & & 0 & I_m \\
                    \end{array}
                  \right]
    $
\EndIf

\For{$i=1:d_{D}-2$}
\If{$\sigma$ has a consecution at $i$}
    \State  $ \mathbb{H}_i := \left[
             \begin{array}{c|c}  
             H_{i} & \\
             \hline
             & H_{i}'
             \end{array}
         \right], 
$ where \\  $H_i := \left[
             \begin{array}{cc}
             	0 & \mathbb{H}_{i-1}(1,1:i+1)   \\
             	-I_n & P_{d_A-i-1}\mathbb{H}_{i-1}(1,1:i+1)   \\ 
             	0 & \mathbb{H}_{i-1}(2:i+1,1:i+1) 
             	 \end{array}
         \right]$ \\
         \vspace{0.3cm}
        $H_i':= \left[
             	\begin{array}{cc}
             	   0 & \mathbb{H}_{i-1}(i+2,2+i:2i+2) \\ 
             	   -I_p & Q_{d_D-i-1}\mathbb{H}_{i-1}(i+2,2+i:2i+2) \\
             	   0 & \mathbb{H}_{i-1}(3+i:2i+2,2+i:2i+2)
             \end{array}
         \right]$ \\

\Else
    \State $\mathbb{H}_i= \left[
             \begin{array}{c|c}  
             H_{i} & \\
             \hline
             & H_{i}'
             \end{array}
         \right],
$  where \\
    $H_i := \left[
              \begin{array}{cc}
              	I_n & \lambda\mathbb{H}_{i-1}(1,1:i+1)   \\ 
              	0 & \mathbb{H}_{i-1}(1:i+1,1:i+1) 
              	\end{array}	 
              	\right]$ \\ \vspace{0.3cm}
              $H_{i}'=\left[
              \begin{array}{cc}
              	 I_m & \lambda \mathbb{H}_{i-1}(i+2,i+2:2i+2)\\
              	 0 & \mathbb{H}_{i-1}(i+2:2i+2,i+2:2i+2)
              \end{array}	 
              	\right]
$  \\
 
\EndIf
\EndFor

   \For{$i = d_{D}-1:d_{A}-2$}

\If{$\sigma$ has a consecution at $i$}
    \State $\mathbb{H}_i := \left[
           \begin{array}{cc}
             0 & \mathbb{H}_{i-1}(1,:)   \\
             -I_{n} & P_{d_{A}-i-1} \mathbb{H}_{i-1}(1, :)    \\
             0 & \mathbb{H}_{i-1}(2:d_{D}+i+1, :)  \\
           \end{array}
         \right]$

\Else
    \State $\mathbb{H}_i :=  \left[
       \begin{array}{cc}
         I_n & \lam \mathbb{H}_{i-1}(1, :)  \\
          0 & \mathbb{H}_{i-1}   \\
       \end{array}
     \right]
   $

\EndIf

\EndFor

\end{algorithmic}
\end{algorithm}

Based on the construction in  Algorithms~\ref{alg3} and  \ref{alg4} we have the following properties.
\begin{lemma}\label{cohiani}
Let $S(\lam)$ be as in (\ref{rosmatrix}) with $A(\lam) = \sum_{i=0}^{d_{A}}\lam^{i}A_i\in \mathbb C[\lam]^{n, n}$ and $D(\lam) = \sum_{j=0}^{d_{D}}\lam^{i}D_i\in \mathbb C[\lam]^{p, m}$ with $d_{A} > d_{D}$. Let $\sig: \{0, 1, \ldots, d_{A}-1\} \rightarrow \{1, 2, \ldots, d_{A}\}$ be a bijection. Let $\{\mathbb{N}_0, \mathbb{N}_1, \ldots, \mathbb{N}_{d_{A}-2}\} $ and $\{\mathbb{H}_0, \mathbb{H}_1, \ldots, \mathbb{H}_{d_{A}-2}\}$ be as in Algorithms~\ref{alg3} and  \ref{alg4}, respectively. Consider the sequence $\{\mathbb{W}_i\}_{i=0}^{d_{A}-2}$  of block partitioned matrices constructed in Algorithm~\ref{alg1}. Then we have the following.
\begin{itemize}
\item[(a)] For $0\leq i \leq d_{A}-2$, and $1 \leq j\leq i+2, $ the number of columns of $\mathbb{N}_i(:, j)$ is equal to the number of rows of $\mathbb{W}_i(j, :) $ so that the product $\mathbb{N}_i(:, j)\mathbb{W}_i(j, :)$ is well defined.

\item[(b)] For $0\leq i \leq d_{A}-2$, and $1 \leq j\leq i+2, $ the number of columns of $\mathbb{W}_i(:, j)$ is equal to the number of rows of $\mathbb{H}_i(j, :)$  so that the product $\mathbb{W}_i(:, j)\mathbb{H}_i(j, :)$ is well defined.

\item[(c)] For $i = 0:d_{D}-2$, the size of $\mathbb{N}_i$ is $\left[\left(n + n \mathfrak{c}(\sig(0:i)) + n \mathfrak{i}(\sig(0:i)) \right) +\left(p + p \mathfrak{c}(\sig(0:i)) + m \mathfrak{i}(\sig(0:i))\right) \right] \times \left[\left(n + n \mathfrak{c}(\sig(0:i)) + n \mathfrak{i}(\sig(0:i)) \right)+\left(p + p \mathfrak{c}(\sig(0:i)) + m \mathfrak{i}(\sig(0:i)) \right)\right]$ and  for $i= d_{D}-1: d_{A}-2$ the size of $\mathbb{N}_i$ is
\begin{small}
\begin{eqnarray*}
&&\left[\left(n+n \mathfrak{c}(\sig(0:i)) + n \mathfrak{i}(\sig(0:i))\right) + \left(p+p \mathfrak{c}(\sig(0:d_D-2)) + m \mathfrak{i}(\sig(0:d_D-2))\right)\right]  \times \\
&&
\left[\left(n+n \mathfrak{c}(\sig(0:i)) + n \mathfrak{i}(\sig(0:i))\right) + \left(p + p \mathfrak{c}(\sig(0:d_D-2)) + m \mathfrak{i}(\sig(0:d_D-2))\right)\right].
\end{eqnarray*}
\end{small}

\item[(d)] For $i= 0:d_{D}-2$, the size of $\mathbb{H}_i$ is $\left[\left(n + n \mathfrak{c}(\sig(0:i)) + n \mathfrak{i}(\sig(0:i))\right) + \left(m + p \mathfrak{c}(\sig(0:i)) + m \mathfrak{i}(\sig(0:i))\right)\right] \times \left[\left(n + n \mathfrak{c}(\sig(0:i)) + n \mathfrak{i}(\sig(0:i)) \right) + \left(m + p \mathfrak{c}(\sig(0:i)) + m \mathfrak{i}(\sig(0:i)) \right)\right]$ and  for $i= d_{D}-1: d_{A}-2$ the size of $\mathbb{H}_i$ is
\begin{small}
\begin{eqnarray*}
&&\left[\left(n+n \mathfrak{c}(\sig(0:i)) + n \mathfrak{i}(\sig(0:i))\right) + \left(m+p \mathfrak{c}(\sig(0:d_D-2)) + m \mathfrak{i}(\sig(0:d_D-2))\right)\right]  \times \\
&&
\left[\left(n+n \mathfrak{c}(\sig(0:i)) + n \mathfrak{i}(\sig(0:i))\right) + \left(m + p \mathfrak{c}(\sig(0:d_D-2)) + m \mathfrak{i}(\sig(0:d_D-2))\right)\right].
\end{eqnarray*}
\end{small}

 \item[(e)] The matrix polynomials $\mathbb{N}_i$ and $\mathbb{H}_i$ are unimodular with $\det(\mathbb{N}_i) = \pm 1$ and $\det(\mathbb{H}_i) = \pm 1$.
\end{itemize}
\end{lemma}

\begin{proof}
The proof follows directly from Lemma $4.2$ given in \cite{DeDM12a}.
\end{proof}

\begin{remark}\label{rem:equaldegrees2}{\rm 
Let $S(\lam)$ be an $(n+p) \times (n+m)$ RSMP as in (\ref{rosmatrix}) with $A(\lam) = \sum\limits_{i=0}^{d_A}\lam^{i}A_i\in \C[\lam]^{n, n}$ and $ D(\lam) = \sum \limits_{i=0}^{d_D}\lam^{i}D_{i} \in \C[\lam]^{p, m}$ with $d_A= d_D$. For this case Algorithms~\ref{alg3} and \ref{alg4} can be applied as well and will stop at $i= 0:d_{D}-2$ and Lemma~\ref{cohiani} also holds for $i= 0:d_{D}-2$.  }
\end{remark}

Since our algorithms work for $d_{A}\geq d_{D}$, in the following we express our results directly for the case $d_{A}\geq d_{D}$.

\begin{definition}
Let $S(\lam)$ be an $(n+p) \times (n+m)$ RSMP as in (\ref{rosmatrix}) with $A(\lam) = \sum\limits_{i=0}^{d_A}\lam^{i}A_i\in \C[\lam]^{n, n}$ and $ D(\lam) = \sum \limits_{i=0}^{d_D}\lam^{i}D_{i} \in \C[\lam]^{p, m}$ with $d_A\geq d_D$. Let $\sig : \{0, 1, \ldots, d_{A}-1\} \rightarrow \{1, 2, \ldots, d_{A}\}$ be a bijection and let $\mathbb{N}_{d_{A}-2}$ and $\mathbb{H}_{d_{A}-2}$,respectively, be the last matrices of the sequences constructed by Algorithms~\ref{alg3} and \ref{alg4}. 
Then
\begin{itemize}
\item[(a)] $\mathcal{U}_{\sig}(\lam) := \mathbb{N}_{d_{A}-2}$ is the left unimodular equivalence matrix associated with $S(\lam)$ and $\sig. $

\item[(a)] $\mathcal{V}_{\sig}(\lam) := \mathbb{H}_{d_{A}-2}$ is the right unimodular equivalence matrix associated with $S(\lam)$ and $\sig. $
\end{itemize}
\end{definition}

\begin{lemma}\label{rbnihiwi}
Let $S(\lam)$ be an $(n+p) \times (n+m)$ RSMP as in (\ref{rosmatrix}) with $A(\lam) = \sum\limits_{i=0}^{d_A}\lam^{i}A_i\in \C[\lam]^{n, n}$, $D(\lam) = \sum \limits_{i=0}^{d_D}\lam^{i}D_{i} \in \C[\lam]^{p, m}$ and $d_A \geq d_D$. Let $\sig  : \{0, 1, \ldots, d_{A}-1\} \rightarrow \{1, 2, \ldots, d_{A}\}$ be a bijection and let $\{\mathbb{W}_i\}_{i=0}^{d_{A}-2}, \{\mathbb{N}_i\}_{i=0}^{d_{A}-2}$, $\{\mathbb{H}_i\}_{i=0}^{d_{A}-2}$ be the sequences of block matrices constructed in Algorithms~\ref{alg1}, \ref{alg3}, and \ref{alg4}, respectively. Also consider the numbers $\alpha_{i} = p \mathfrak{c}(\sig(0:i)) + m \mathfrak{i}(\sig(0:i))$, $\beta_{i} = p \mathfrak{c}(\sig(i)) + m \mathfrak{i}(\sig(i))$ and $\alpha_{i}' = n \mathfrak{c}(\sig(0:i)) + n \mathfrak{i}(\sig(0:i))$, $\beta_{i}' = n \mathfrak{c}(\sig(i)) + n \mathfrak{i}(\sig(i))$. Note that $\beta_{i} = p$ if $\sig$ has a consecution at $i$ and $\beta_{i} = m$, if $\sig$ has an inversion at $i$, for $i = 0, 1, 2, \ldots, d_{A}-2$. Then we have the following.
\begin{itemize}
\item[(a)] For $1 \leq i \leq d_{A}-2$, 
{\scriptsize $$\mathbb{N}_i \left(\left[
              \begin{array}{cc|cc}
                \lam P_{d_{A}-i-2} &  & & \\
                 & \lam I_{\alpha_i'} & & \\
                \hline
                 &  & \lam Q_{d_{D}-i-2}& \\
                 &  &  & \lam I_{\alpha_i} \\
              \end{array}
            \right]
 -\mathbb{W}_i \right)\mathbb{H}_i 
   = \left[
                      \begin{array}{cc|cc}
                         I_{\beta_i}  &  & 0 & 0  \\
                         & T_1 & 0 & -B \\
                        \hline
                         0 & 0  & I_{\beta_i}  & \\
                        0 &  C  &  &
                         K_1 \\
                      \end{array}
                    \right], $$}
 where $T_1 = N_{i-1} \left( \left[
		\begin{array}{cc}
			\lam P_{d_{A}-i-1} &  \\
			& \lam I_{\alpha_{i-1}} \\
		\end{array}
		\right] -W_{i-1} \right)H_{i-1}$ and $$K_1= N_{i-1}' \left( \left[
		\begin{array}{cc}
			\lam Q_{d_{D}-i-1} &  \\
			& \lam I_{\alpha_{i-1}} \\
		\end{array}
		\right] -W_{i-1}' \right)H_{i-1}'. $$
  
 \item[(b)] For $i = 0$
 {\scriptsize $$\mathbb{N}_0 \left(\left[
              \begin{array}{cc|cc}
                \lam P_{d_{A}-2} &  & &  \\
                 & \lam I_{\alpha_0'} &  &  \\
                \hline
                  &  & \lam Q_{d_{D}-2} &  \\
                 &  &  & \lam I_{\alpha_0}  \\
              \end{array}
            \right]
 -\mathbb{W}_0 \right)\mathbb{H}_0 = \left[
                     \begin{array}{cc|cc}
                       I_{\beta_0'} &  & 0 & 0  \\
                        & A(\lam) & 0 & -B \\
                       \hline
                       0 & 0 & I_{\beta_0} &  \\
                       0 & C & & D(\lam) \\
                     \end{array}
                   \right].
 $$}
\end{itemize}
\end{lemma}
\begin{proof}
\begin{itemize}
\item[(a)] Suppose that $\sig$ has a consecution at $i$. Let us write $\mathbb{N}_i=\left[ \begin{array}{c|c} N_i & \\ \hline & N'_i \end{array} \right],$ where 

$N_i = \left[ \begin{array}{cc} I_n & 0  \\ \lambda \mathbb{N}-{i-1}(i:i+1,1) & \mathbb{N}_{i-1}(1:i+1,1:i+1) \end{array} \right]$  and \\ 
$N'_i = \left[ \begin{array}{cc} I_p & 0 \\ \mathbb{N}_{i-1}(2+i:2i+2,i+2) & \mathbb{N}_{i-1}(2+i:2i+2,2+i:2i+2) \end{array} \right]$,

and 

$\mathbb{H}_i=\left[ \begin{array}{c|c} H_i & \\ \hline & H'_i \end{array} \right],$ where

$H_i = \left[ \begin{array}{cc} 0 & \mathbb{H}_{i-1}(1,1:i+1) \\ -I_m & P_{d_A-i-1}\mathbb{H}_{i-1}(1,1:i+1) \\ 0 & \mathbb{H}_{i-1}(2:i+1,1:i+1) \end{array} \right],$ \\
$H'_i = \mathbb{H}_i := \left[ \begin{array}{cc} 0 & \mathbb{H}_{i-1}(i+2,2+i:2i+2) \\ -I_p & Q_{d_D-i-1}\mathbb{H}_{i-1}(i+2,2+i:2i+2) \\ 0 & \mathbb{H}_{i-1}(3+i:2i+2,2+i:2i+2) \end{array} \right], $

and 
 
$\mathbb{W}_i=\left[ \begin{array}{c|c} W_i & W_{12} \\ \hline W_{21} & W_i' \end{array} \right]$, where
 
$W_i = \left[ \begin{array}{ccc} -A_{i+1} & I_{n} &  0 \\ \mathbb{W}_{i-1}(1:i+1,1) & 0 & \mathbb{W}_{i-1}(1:i+1,2:i+1) \end{array} \right], $ \\
$W_{12} = \left[ \begin{array}{ccc}  0 & 0 &  0  \\ \mathbb{W}_{i-1}(1:i+1, i+2) & 0 & \mathbb{W}_{i-1}(1:i+1, i+3:2i+2) \end{array} \right], $ \\
$W_{21} = \left[ \begin{array}{ccc}	0 & 0 &  0 \\ \mathbb{W}_{i-1}(2+i:2i+2, 1) & 0 & \mathbb{W}_{i-1}(2+i:2i+2,2:i+1)  \end{array} \right], $ \\
$W_i' = \left[ \begin{array}{ccc} -D_{i+1} & I_{p} &  0 \\ \mathbb{W}_{i-1}(2+i:2i+2,i+2) & 0 & \mathbb{W}_{i-1}(2+i:2i+2, i+3:2i+2)  \end{array} \right].$

Now,
\begin{eqnarray*}
& \mathbb{N}_i \left(\left[
              \begin{array}{cc|cc}
                \lam P_{d_{A}-i-2} &  & & \\
                 & \lam I_{\alpha_i'} & & \\
                \hline
               &  & \lam Q_{d_{D}-i-2} & \\
                 &  & & \lam I_{\alpha_i} \\
              \end{array}
            \right]
 -\mathbb{W}_i \right)\mathbb{H}_i
\end{eqnarray*}{\scriptsize
$$= \left[ {\begin{array}{c|c}
N_i &  \\
\hline
 & N_{i}' \\
\end{array}}\right] \left( \left[ {\begin{array}{cc|cc}
 \lambda P_{d_{A}-i-2} & &  \\
 & \lambda I_{\alpha_i'} &  \\
\hline
 &  & \lam Q_{d_{D}-i-2} & \\
 &  & & \lam I_{\alpha_i} \\
\end{array}}\right] - \left[ {\begin{array}{c|c}
 W_i & W_{12} \\ \hline
 W_{21} & W_i'
 \end{array}}\right]\right) \left[ {\begin{array}{c|c}
                                           H_i &  \\
                                           \hline
                                             & H_{i}' \\
                                       \end{array}}\right], $$}
{\scriptsize $$ = \left[ {\begin{array}{c|c}
N_i &  \\
\hline
 &  N_{i}' \\
\end{array}}\right] \left[
                      \begin{array}{c|c}
                          \left[
                          \begin{array}{cc}
                            \lam P_{d_{A}-i-2} &  \\
                             & \lam I_{\alpha_i'} \\
                          \end{array}
                        \right] -W_i  & -W_{12} \\
                        \hline
                        -W_{21} & \left[
                          \begin{array}{cc}
                            \lam Q_{d_{D}-i-2} &  \\
                             & \lam I_{\alpha_i} \\
                          \end{array}
                        \right] -W_i' \\
                      \end{array}
                    \right]
 \left[ {\begin{array}{c|c}
                                           H_i &  \\
                                           \hline
                                             & H_{i}' \\
                                       \end{array}}\right] $$}
{\scriptsize $$ = \left[ {\begin{array}{c|c}
N_i &  \\
\hline
 &  N_{i}' \\
\end{array}}\right] \left[
                      \begin{array}{c|c}                     \left( \left[
                          \begin{array}{cc}
                            \lam P_{k-i-2} &  \\
                             & \lam I_{\alpha_i} \\
                          \end{array}
                        \right] -W_i \right)H_i & -W_{12}H_i' \\
                        \hline
                        -W_{21}H_i & \left( \left[
                          \begin{array}{cc}
                            \lam Q_{d_{D}-i-2} &  \\
                             & \lam I_{\alpha_i} \\
                          \end{array}
                        \right] -W_i' \right)H_i' \\
                      \end{array}
                    \right]$$}
{\scriptsize $$ = \left[
                      \begin{array}{c|c}
                          N_i \left( \left[
                          \begin{array}{cc}
                            \lam P_{d_{A}-i-2} &  \\
                             & \lam I_{\alpha_i} \\
                          \end{array}
                        \right] -W_i \right)H_i &  -N_iW_{12}H_i'\\
                        \hline
                        -N_i'W_{21}H_i & N_i'\left( \left[
                          \begin{array}{cc}
                            \lam Q_{d_{D}-i-2} &  \\
                             & \lam I_{\alpha_i} \\
                          \end{array}
                        \right] -W_i' \right)H_i' \\
                      \end{array}
                    \right].$$}
                    
Now, for $1 \leq i \leq d_D-2$, we have
\begin{align*}
	N_i'W_{21}H_i\\ =& \left[ \begin{array}{cc} I_p & 0 \\ \mathbb{N}_{i-1}(2+i:2i+2,i+2) & \mathbb{N}_{i-1}(2+i:2i+2,2+i:2i+2) \end{array} \right] \\ 
          & \left[ \begin{array}{ccc}	0 & 0 &  0 \\ \mathbb{W}_{i-1}(2+i:2i+2, 1) & 0 & \mathbb{W}_{i-1}(2+i:2i+2,2:i+1)  \end{array} \right] \\        
                         &  \left[ \begin{array}{cc} 0 & \mathbb{H}_{i-1}(1,1:i+1) \\ -I_n & P_{d_A-i-1}\mathbb{H}_{i-1}(1,1:i+1) \\ 0 & \mathbb{H}_{i-1}(2:i+1,1:i+1) \end{array} \right] \\
                         = & \left[\begin{array}{cc} 0 & 0 \\ 0 & Y \end{array} \right],
\end{align*}                    
where
$Y=\mathbb{N}_{i-1}(2+i:2i+2,2+i:2i+2)\mathbb{W}_{i-1}(3:i+3,1)\mathbb{H}_{i-1}(1,i+1)+\mathbb{N}_{i-1}(2+i:2i+2,2+i:2i+2)\mathbb{W}_{i-1}(3:i+3,2:i+1)\mathbb{H}_{i-1}(2:i+1,1:i+1)$
and
$\left[\begin{array}{cc} 0 & 0 \\ 0 & Y \end{array} \right]
= (e_{2+i}e^T_{2+i})\otimes C$,

and
\begin{align*}
	N_iW_{12}H_i' \\  
	      = & \left[ \begin{array}{cc} I_n & 0  \\ \lambda \mathbb{N}-{i-1}(i:i+1,1) & \mathbb{N}_{i-1}(1:i+1,1:i+1) \end{array} \right] \\
             & \left[ \begin{array}{ccc}  0 & 0 &  0  \\ \mathbb{W}_{i-1}(1:i+1, i+2) & 0 & \mathbb{W}_{i-1}(1:i+1, i+3:2i+2) \end{array} \right] \\ 
               & \left[ \begin{array}{cc} 0 & \mathbb{H}_{i-1}(i+2,2+i:2i+2) \\ -I_p & P_{d_D-i-1}\mathbb{H}_{i-1}(i+2,2+i:2i+2) \\ 0 & \mathbb{H}_{i-1}(3+i:2i+2,2+i:2i+2) \end{array} \right] \\
              = & \left[\begin{array}{cc} 0 & 0 \\ 0 & X \end{array}\right],
\end{align*}
where
$X= \mathbb{N}_{i-1}(1:i+1,1:i+1)\mathbb{W}_{i-1}(1:i+1,i+2)\mathbb{H}_{i-1}(i+2,2+i:2i+2)+\mathbb{N}_{i-1}(1:i+1,1:i+1)\mathbb{W}_{i-1}(1:i+1,i+3:2i+2)\mathbb{H}_{i-1}(3+i:2i+2,2+i:2i+2)  $
and
 $\left[\begin{array}{cc} 0 & 0 \\ 0 & X \end{array}\right]
= (e_{2+i}e^T_{2+i})\otimes B$.

For $i=d_D-1:d_{A}-2$, the blocks labelled $W_i$, $N_i$, and $H_i$ are constant. If $\sigma$ has a consecution at $i$ then one block row is added at the top of $\mathbb{W}_i$ and,  hence, one block row is added to the block labelled $W_{12}$ and the block labelled $W_{21}$ remains unchanged. Similarly, if $\sigma$ has an inversion at $i$, one block column is added at the extreme left of the matrix $\mathbb{W}_i$ and, hence, one block column is added to the block labelled $W_{21}$ but the block labelled $W_{12}$ remains unchanged.

Hence, 
\[
N_iW_{12}H'_i=(e_{1+i}e_{d_D})\otimes B,\
N'_iW_{21}H_i=(e_{d_D}e_{1+i})\otimes C.
\]
By Lemma $4.4$ given in \cite{DeDM12a}  and the previous Lemmas in this section, we have
{\scriptsize \begin{eqnarray*}
& \mathbb{N}_i \left(\left[
	\begin{array}{cc|cc}
		\lam P_{d_{A}-i-2} &  & & \\
		& \lam I_{\alpha_i'} & & \\
		\hline
		&  & \lam Q_{d_{D}-i-2} & \\
		&  & & \lam I_{\alpha_i} \\
	\end{array}
	\right]
	-\mathbb{W}_i \right)\mathbb{H}_i
  =  \left[
	\begin{array}{cc|cc}
		I_{\beta_i'}  &  & 0 & 0 \\
		& T_1 & 0 & -B \\
		\hline
		0 &  0 &
		I_{\beta_i}  &     \\
		0 & C & & K_1
	\end{array}
	\right], \end{eqnarray*}}
where $T_1 = N_{i-1} \left( \left[
		\begin{array}{cc}
			\lam P_{d_{A}-i-1} &  \\
			& \lam I_{\alpha_{i-1}} \\
		\end{array}
		\right] -W_{i-1} \right)H_{i-1}$ and $$K_1= N_{i-1}' \left( \left[
		\begin{array}{cc}
			\lam Q_{d_{D}-i-1} &  \\
			& \lam I_{\alpha_{i-1}} \\
		\end{array}
		\right] -W_{i-1}' \right)H_{i-1}'. $$

\item[(b)] Suppose that $\sig$ has a consecution at $0$. Then we have $\alpha_0  = p,  \alpha_0' = n$ and $\beta_0 = p, \beta_0' =n$. Now
\begin{eqnarray*}
& \mathbb{N}_0 \left(\left[
              \begin{array}{cc|cc}
                \lam P_{d_{A}-2} &  & &  \\
                 & \lam I_{\alpha_0'} &  &  \\
                \hline
                  &  & \lam Q_{d_{D}-2} &  \\
                 &  &  & \lam I_{\alpha_0}  \\
              \end{array}
            \right]
 -\mathbb{W}_0 \right)\mathbb{H}_0 \\ & = \mathbb{N}_0 \left(\left[
              \begin{array}{cc|cc}
                \lam P_{d_{A}-2} &  & &  \\
                 & \lam I_{n} &  &  \\
                \hline
                  &  & \lam Q_{d_{D}-2} &  \\
                 &  &  & \lam I_{p}  \\
              \end{array}
            \right]
 -\left[
    \begin{array}{cc|cc}
      -A_1 & I_n & 0 & 0 \\
      -A_0 & 0 & -B & 0 \\
      \hline
      0 & 0 & -D_1 & I_p\\
      -C & 0 & -D_0 & 0 \\
    \end{array}
  \right]
  \right)\mathbb{H}_0  \\
& = \left[
     \begin{array}{cc|cc}
       I_n & 0 & & \\
       \lam I_n & I_n &  \\
       \hline
        &  & I_p & 0 \\
       &  & \lam I_p & I_p  \\
     \end{array}
   \right] \left[
             \begin{array}{cc|cc}
               \lam P_{d_{A}-2}+A_1 & -I_n & 0 & 0 \\
               A_0 & \lam I_{n} & -B & 0 \\
               \hline
               0 & 0 & \lam Q_{d_{D}-2}+D_1 & -I_p  \\
              C & 0 & D_0 & \lam I_{p}  \\
             \end{array}
           \right] \\
          & \hspace{1.5cm} \left[
                     \begin{array}{cc|cc}
                       0 & I_n & &  \\
                       -I_n & P_{d_{A}-1} & & \\
                       \hline
                       & & 0 & I_m \\
                       & & -I_p & Q_{d_{D}-1}  \\
                     \end{array}
                   \right] \\
& = \left[
                     \begin{array}{cc|cc}
                       I_{\beta_0'} &  & 0 & 0 \\
                        & A(\lam) & 0 & -B \\
                       \hline
                       0 & 0 & I_{\beta_0} &  \\
                       0 & C & & D(\lam) \\
                     \end{array}
                   \right],
\end{eqnarray*}
 since $\beta_0' = n$ and
\[
\left[
           \begin{array}{cc}
             I_n & 0 \\
             \lam I_n & I_n \\
           \end{array}
         \right] \left[
                   \begin{array}{cc}
                     P_{d_{A}-1} & -I_n \\
                     A_0 & \lam I_{n} \\
                   \end{array}
                 \right] \left[
                           \begin{array}{cc}
                             0 & I_n \\
                             -I_n & P_{d_{A}-1} \\
                           \end{array}
                         \right] = \left[
                                     \begin{array}{cc}
                                       I_{n} &  \\
                                        & A(\lam) \\
                                     \end{array}
                                   \right], 
\]
since $\beta_0 = p$  and
\[
\left[
           \begin{array}{cc}
             I_p & 0 \\
             \lam I_p & I_p \\
           \end{array}
         \right] \left[
                   \begin{array}{cc}
                     Q_{d_{D}-1} & -I_p \\
                     D_0 & \lam I_{p} \\
                   \end{array}
                 \right] \left[
                           \begin{array}{cc}
                             0 & I_m \\
                             -I_p & Q_{d_{D}-1} \\
                           \end{array}
                         \right] = \left[
                                     \begin{array}{cc}
                                       I_{p} &  \\
                                        & D(\lam) \\
                                     \end{array}
                                   \right].
\]
\end{itemize}
\end{proof}

After these preparations we get our main theorem.
\begin{theorem} \label{LinT}
Let $S(\lam)$ be an $(n+p) \times (n+m)$ RSMP as  in (\ref{rosmatrix}) with $A(\lam) = \sum\limits_{i=0}^{d_A}\lam^{i}A_i\in \C[\lam]^{n, n}$ and $D(\lam) = \sum \limits_{i=0}^{d_D}\lam^{i}D_{i} \in \C[\lam]^{p, m},$ and $d_A \geq d_D$.  For any bijection $\sig  : \{0, 1, \ldots, d_{A}-1\} \rightarrow \{1, 2, \ldots, d_{A}\}$, the Fiedler pencil $\mathbb{L}_{\sig}(\lam)$ is a linearization for $S(\lam)$.
\end{theorem}
\begin{proof}
Recall that $\mathcal{U}_{\sig}(\lam) = \mathbb{N}_{d_{A}-2}$ and $\mathcal{V}_{\sig}(\lam) = \mathbb{H}_{d_{A}-2}$, and that $\mathbb{M}_{\sig} = \mathbb{W}_{d_{A}-2}$. Let $P_k$ and $Q_k$ denote the degree $k$ Horner shifts of $A(\lambda)$ and $D(\lambda),$ respectively. For $i = d_{A}-2$, then by  Lemma \ref{rbnihiwi} $(a)$ we have
{ \[ 
\mathcal{U}_{\sig}(\lam) \mathbb{L}_{\sig}(\lam) \mathcal{V}_{\sig}(\lam) = \mathbb{N}_{d_{A}-2}\left(\lam \left[
                                  \begin{array}{cc|cc}
                                P_0 &  &  & \\
                                & I_{\alpha'_{d_A-2}} & & \\
                                         \hline
                                              &  & Q_0 & \\
                      &  &  & I_{\alpha_{d_D-2}} \\
                       \end{array}
                                         \right]
- \mathbb{W}_{d_{A}-2}\right) \mathbb{H}_{d_{A}-2} 
\]}
{\[ = \left[
                      \begin{array}{cc|cc}
                         I_{\beta'_{d_A-2}}  &  & 0 & 0 \\
                         & T_3  & 0 & -B\\
                        \hline
                      0 & 0 & I_{\beta_{d_D-2}} &  \\
                     0 & C &  &  K_3 \\
                         \end{array}
                    \right],
\]
} where $T_3 = N_{d_A-3} \left( \left[
                          \begin{array}{cc}
                            \lam P_{1} &  \\
                             & \lam I_{\alpha'_{d_A-3}} \\
                          \end{array}
                        \right] -W_{d_A-3} \right)H_{d_A-3}$ and 
$$K_3 =  \begin{array}{cc}
                          N_{d_D-3}' \left( \left[
                          \begin{array}{cc}
                            \lam Q_{1} &  \\
                             & \lam I_{\alpha_{d_D-3}} \\
                          \end{array}
                        \right] -W_{d_D-3}' \right)H_{d_D-3}'  \\
                         \end{array}.$$ 
Again for $i = d_{A}-3$, we have 
{\scriptsize \[\mathcal{U}_{\sig}(\lam) \mathbb{L}_{\sig}(\lam) \mathcal{V}_{\sig}(\lam) = \left[
                                    \begin{array}{cc|c}
                                      I_{\beta'_{d_A-2}} &  &   \\
                                       & I_{\beta'_{d_A-3}} &    \\
                                       &  & \mathbb{N}_{d-4}\left( \left[
                              \begin{array}{cc|cc}
                              \lam P_2 &  &  &   \\
                               & \lam I_{\alpha'_{d_A-4}} &  & \\
                               \hline
                                &  & \lam Q_2 & \\
                                 & &  &  \lam I_{\alpha_{d_D-4}} \\
                                \end{array}
                                 \right]
- \mathbb{W}_{d-4}\right) \mathbb{H}_{d-4}   \\
                                    \end{array}
                                  \right].
\]}
Applying Lemma \ref{rbnihiwi} $(a)$ for $i = d_{A}-4, d_{A}-5, \ldots, 1 $, we have
{\scriptsize\[
\mathcal{U}_{\sig}(\lam) \mathbb{L}_{\sig}(\lam) \mathcal{V}_{\sig}(\lam) = \left[
                                                                                \begin{array}{cc}
                                                                                 I_{\beta'_{d_A-2}+ \beta'_{d_A-3} + \ldots + \beta'_{1}} &  \\
                                                                                   &  \mathbb{N}_{0}\left( \left[
                              \begin{array}{cc|cc}
                              \lam P_{d_A-2} &  &  \\
                               & \lam I_{\alpha'_{0}} &  \\
                               \hline
                                &  & \lam Q_{d_D-2} \\ & & & \lambda I_{\alpha_0}
                                \end{array}
                                 \right]
- \mathbb{W}_{0}\right) \mathbb{H}_{0}\\
                                                                                \end{array}
                                                                              \right].
\]}
Applying Lemma \ref{rbnihiwi} $(b)$, we have
{\scriptsize\[
\mathcal{U}_{\sig}(\lam) \mathbb{L}_{\sig}(\lam) \mathcal{V}_{\sig}(\lam)
= \left[
         \begin{array}{cc|cc}
              I_{\beta'_{d_A-2}+ \beta'_{d_A-2} + \ldots + \beta'_{1}+ \beta'_0} & 0 & 0 & 0 \\
               0 & A(\lam) & 0 & -B \\
                 \hline
              0 & 0 & I_{\beta_{d_D-2}+ \beta_{d_D-2} + \ldots + \beta_{1}+ \beta_0}& 0\\
                       0 & C & 0 & D(\lam) \\
                                              \end{array}
                                    \right].
\]}
Since $I_{\beta'_{d_A-2}+ \beta'_{d_A-3} + \ldots + \beta'_{1} + \beta'_0} = \alpha'_{d_A-2} = n \mathfrak{c}(\sig) + n \mathfrak{i}(\sig)$ and $I_{\beta_{d_D-2}+ \beta_{d_D-3} + \ldots + \beta_{1} + \beta_0} = \alpha_{m-2} = p \mathfrak{c}(\sig) + m \mathfrak{i}(\sig)$, we have
\[
\mathcal{U}_{\sig}(\lam) \mathbb{L}_{\sig}(\lam) \mathcal{V}_{\sig}(\lam) = 
\left[
                          \begin{array}{cc|cc}
                     I_{n \mathfrak{c}(\sig) + n \mathfrak{i}(\sig)} & 0& 0 & 0 \\
                                                 0   & A(\lam) & 0 & -B \\
                                                  \hline
                     0 & 0 & I_{p \mathfrak{c}(\sig) + m \mathfrak{i}(\sig)} & 0\\
                                                 0 & C & 0 & D(\lam) \\
                                            \end{array}
                                             \right],            
                                              \]
which proves that $\mathbb{L}_{\sig}(\lam)$ is a linearization of $S(\lam)$, since $\mathcal{U}_{\sig}(\lam)$ and $\mathcal{V}_{\sig}(\lam)$ are unimodular.
\end{proof}

\begin{corollary}
Let $S(\lam)$  be as in (\ref{rosmatrix}) with $A(\lam) = \sum\limits_{i=0}^{d_A}\lam^{i}A_i\in \C[\lam]^{n, n}$ and $D(\lam) = \sum \limits_{i=0}^{d_D}\lam^{i}D_{i} \in \C[\lam]^{p, m}$ with $d_A\geq d_D$.  Let $\sig : \{0, 1, \ldots, d_{A}-1\} \rightarrow \{1, 2, \ldots, d_{A}\}$ be a bijection and let $\mathcal{U}_{\sig}(\lam)$, $\mathcal{V}_{\sig}(\lam)$, respectively, be the left and right unimodular equivalences associated with $S(\lam)$. Then
\[
\mathcal{U}_{\sig}(\lam) \mathbb{L}_{\sig}(\lam) \mathcal{V}_{\sig}(\lam) = \left[
                              \begin{array}{c|c|c}
                               I_{n \mathfrak{c}(\sig) + n \mathfrak{i}(\sig)} &  \\
                                      \hline                   
                                      & S(\lam) & \\
                                             \hline
                          &  & I_{p \mathfrak{c}(\sig) + m \mathfrak{i}(\sig)} \\
                                                    \end{array}
                                                \right].
\]
\end{corollary}

\begin{remark}
Note that the construction that we have discussed in section~\ref{sec:LFP} to show that the Fiedler pencils of rectangular rational matrix are linearizations, will also work to show that the Fiedler pencils defined in Definition~\ref{def:fiedlerS} for square RSMPs with $B, C$ constant in $\lambda$ are  linearizations. Also, there is another technique,  discussed in \cite{BehB22}, to show that  the Fiedler pencils defined in Definition~\ref{def:fiedlerS} are linearizations for square RSMPs. 
\end{remark}

\section{Conclusion}\label{sec:concl}
The construction of Fiedler pencils for rectangular Rosenbrock  matrix polynomials has been studied and it has been shown that these are linearizations.

\bibliographystyle{plain}
\bibliography{Beh_m}

\begin{thebibliography}{10}

\bibitem{AhmM16}
S.S. Ahmad and V.~Mehrmann.
\newblock Backward errors and pseudospectra for structured nonlinear eigenvalue
  problems.
\newblock {\em OaM}, 10:539--556, 09 2016.

\bibitem{AlaB16}
R.~Alam and N.~Behera.
\newblock Linearizations for rational matrix functions and rosenbrock system
  polynomials.
\newblock {\em {SIAM} J. Matrix Anal. Appl.}, 37:354--380, 2016.

\bibitem{AmpDMZ18}
A.~Amparan, F.M. Dopico, S.~Marcaida, and I.~Zaballa.
\newblock Strong linearizations of rational matrices.
\newblock {\em SIAM J. Matrix Anal. Appl.}, 39(4):1670--1700, 2018.

\bibitem{AmpDMZ21}
A.~Amparan, F.M. Dopico, S.~Marcaida, and I.~Zaballa.
\newblock On minimal bases and indices of rational matrices and their
  linearizations.
\newblock {\em Linear Algebra Appl.}, 623:14--67, 2021.

\bibitem{AntV04}
E.~N. Antoniou and S.~Vologiannidis.
\newblock A new family of companion forms of polynomial matrices.
\newblock {\em Electron. J. Linear Algebra}, 11:78--87, 2004.

\bibitem{Ant05}
A.C. Antoulas.
\newblock {\em Approximation of large-scale dynamical systems}, volume~6 of
  {\em Advances in Design and Control}.
\newblock Society for Industrial and Applied Mathematics (SIAM), Philadelphia,
  PA, 2005.

\bibitem{AntBG10}
A.C. Antoulas, C.A. Beattie, and S.~Gugercin.
\newblock Interpolatory model reduction of large-scale dynamical systems.
\newblock In J.~Mohammadpour and K.~Grigoriadis, editors, {\em Efficient
  Modeling and Control of Large-Scale Systems}. Springer-Verlag, 2010.

\bibitem{BauBF14}
U.~Baur, P.~Benner, and L.~Feng.
\newblock Model order reduction for linear and nonlinear systems: a
  system-theoretic perspective.
\newblock {\em Arch. Comput. Meth. Enginrg}, 21(4):331--358, 2014.

\bibitem{Beh14}
N.~Behera.
\newblock {\em Fiedler Linearizations for LTI State-Space Systems and for
  Rational Eigenvalue Problems}.
\newblock {PhD} thesis, Indian Institute of Technology Guwahati, 2014.

\bibitem{BehB22}
N.~Behera and A.~Bist.
\newblock Fiedler linearizations of multivariable statespace systems and its
  associated system matrix.
\newblock {\em arXiv:2207.01324}, 2022.

\bibitem{BetHMST13}
T.~Betcke, N.J. Higham, V.~Mehrmann, C.~Schr{\"o}der, and F.~Tisseur.
\newblock Nlevp: A collection of nonlinear eigenvalue problems.
\newblock {\em ACM Trans. Math. Softw.}, 39:7:1--7:28, 2013.

\bibitem{BinMMS15}
A.~Binder, V.~Mehrmann, A.~Miedlar, and P.~Schulze.
\newblock A \texttt{Matlab} toolbox for the regularization of descriptor
  systems arising from generalized realization procedures.
\newblock Preprint 24--2015, TU Berlin, Institut f\"ur Mathematik, 2015.

\bibitem{BueD11}
M.I. Bueno, F.~De~Ter{\'a}n, and F.M. Dopico.
\newblock Recovery of eigenvectors and minimal bases of matrix polynomials from
  generalized {F}iedler linearizations.
\newblock {\em SIAM J. Matrix Anal. Appl.}, 32(2):463--483, 2011.

\bibitem{ByeMX08}
R.~Byers, V.~Mehrmann, and H.~Xu.
\newblock Trimmed linearizations for structured matrix polynomials.
\newblock {\em Linear Algebra Appl.}, 429:2373--2400, 2008.

\bibitem{CamKM12}
S.~L. {Campbell}, P.~{Kunkel}, and V.~{Mehrmann}.
\newblock Regularization of linear and nonlinear descriptor systems.
\newblock In L.~T. {Biegler}, S.~L. {Campbell}, and V.~{Mehrmann}, editors,
  {\em Control and Optimization with Differential-Algebraic Constraints},
  number~23 in Advances in Design and Control, pages 17--36, 2012.

\bibitem{ConPV89}
C.~Conca, J.~Planchard, and M.~Vanninathan.
\newblock Existence and location of eigenvalues for fluid-solid structures.
\newblock {\em Comput. Methods Appl. Mech. Engrg.}, 77(3):253--291, 1989.

\bibitem{DasAlm19}
R.~K. Das and R.~Alam.
\newblock Affine spaces of strong linearizations for rational matrices and the
  recovery of eigenvectors and minimal bases.
\newblock {\em Linear Algebra Appl.}, 569:335--368, 2019.

\bibitem{DeDM12a}
F.~{De~Ter\'{a}n}, F.~M. Dopico, and D.~S. Mackey.
\newblock Fiedler companion linearizations for rectangular matrix polynomials.
\newblock {\em Linear Algebra Appl.}, 437:957--991, 2012.

\bibitem{DeTDP13}
F.~De~Ter{\'a}n, F.M. Dopico, and J.~P{\'e}rez.
\newblock Condition numbers for inversion of {F}iedler companion matrices.
\newblock {\em Linear Algebra Appl.}, 439(4):944--981, 2013.

\bibitem{DeTDM09}
F.~De~Ter{\'a}n, F.~M.~Dopico, and D.S. Mackey.
\newblock Fiedler companion linearizations and the recovery of minimal indices.
\newblock {\em {SIAM} J. Matrix Anal. Appl.}, 31(4):2181--2204, 2009/10.

\bibitem{DemK86}
J.~W. Demmel and B.~K{\aa}gstr\"om.
\newblock Stably computing the {K}ronecker structure and reducing subspaces of
  singular pencils {$A-\lambda B$} for uncertain data.
\newblock In J.~Cullum and R.A. Willoughby, editors, {\em Large Scale
  Eigenvalue Problems}, pages 283--323. Elsevier, North-Holland, 1986.

\bibitem{DopMQV20}
F.M. Dopico, S.~Marcaida, M.C. Quintana, and P.~Van Dooren.
\newblock Local linearizations of rational matrices with application to
  rational approximations of nonlinear eigenvalue problems.
\newblock {\em Linear Algebra Appl.}, 604:441--475, 2020.

\bibitem{DMQD22}
M.~C.~Quintana F.~M.~Dopico, S.~Marcaida and P.~V. Dooren.
\newblock Linearizations of matrix polynomials viewed as rosenbrock's system
  matrices.
\newblock {\em arXiv:2211.09056, preprint}, 2022.

\bibitem{DMQD23}
M.~C.~Quintana F.~M.~Dopico, S.~Marcaida and P.~V. Dooren.
\newblock Block full rank linearizations of rational matrices.
\newblock {\em Linear Multilinear Algebra}, 71(3):391--421, 2023.

\bibitem{GohLR82}
I.~Gohberg, P.~Lancaster, and L.~Rodman.
\newblock {\em Matrix Polynomials}.
\newblock Academic Press, New York, 1982.

\bibitem{GraHT11}
L.~Grammont, N.J. Higham, and F.~Tisseur.
\newblock A framework for analyzing nonlinear eigenproblems and parametrized
  linear systems.
\newblock {\em Linear Algebra Appl.}, 435(3):623--640, 2011.

\bibitem{GugA04}
S.~Gugercin and A.C. Antoulas.
\newblock A survey of model reduction by balanced truncation and some new
  results.
\newblock {\em Int. J. of Control}, 77:748--766, 2004.

\bibitem{HigMMT06}
N.~J. Higham, D.~S. Mackey, N.~Mackey, and F.~Tisseur.
\newblock Symmetric linearizations for matrix polynomials.
\newblock {\em {SIAM} J. Matrix Anal. Appl.}, 29(1):143--159, 2006.

\bibitem{HigLT07}
N.J. Higham, R.-C. Li, and F.~Tisseur.
\newblock Backward error of polynomial eigenproblems solved by linearization.
\newblock {\em SIAM J. Matrix Anal. Appl.}, 29(4):1218--1241, 2007.

\bibitem{HigMT06}
N.J. Higham, D.S. Mackey, and F.~Tisseur.
\newblock The conditioning of linearizations of matrix polynomials.
\newblock {\em {SIAM} J. Matrix Anal. Appl.}, 28:1005--1028, 2006.

\bibitem{HigMT09}
N.J. Higham, D.S. Mackey, and F.~Tisseur.
\newblock Definite matrix polynomials and their linearizations by definite
  pencils.
\newblock {\em {SIAM} J. Matrix Anal. Appl.}, 31:478--502, 2009.

\bibitem{HilMM04}
A.~Hilliges, C.~Mehl, and V.~Mehrmann.
\newblock On the solution of palindromic eigenvalue problems.
\newblock In {\em Proceedings of the 4th European Congress on Computational
  Methods in Applied Sciences and Engineering (ECCOMAS)}. Jyv\"askyl\"a,
  Finland, 2004.
\newblock CD-ROM.

\bibitem{HwaLWW04}
T.-M. Hwang, W.-W. Lin, W.-C. Wang, and W.~Wang.
\newblock Numerical simulation of three dimensional pyramid quantum dot.
\newblock {\em J. Comput. Phys.}, 196(1):208--232, 2004.

\bibitem{KunM06}
P.~{Kunkel} and V.~{Mehrmann}.
\newblock {\em Differential-Algebraic Equations. Analysis and Numerical
  Solution}.
\newblock Z\"urich: European Mathematical Society Publishing House, 2006.

\bibitem{MacMMM06b}
D.S. Mackey, N.~Mackey, C.~Mehl, and V.~Mehrmann.
\newblock Structured polynomial eigenvalue problems: {G}ood vibrations from
  good linearizations.
\newblock {\em {SIAM} J. Matrix Anal. Appl.}, 28(4):1029--1051, 2006.

\bibitem{MacMMM06a}
D.S. Mackey, N.~Mackey, C.~Mehl, and V.~Mehrmann.
\newblock Vector spaces of linearizations for matrix polynomials.
\newblock {\em {SIAM} J. Matrix Anal. Appl.}, 28(4):971--1004, 2006.

\bibitem{MayA07}
A.~J. Mayo and A.~C. Antoulas.
\newblock A framework for the solution of the generalized realization problem.
\newblock {\em Linear Algebra Appl.}, 425:634--662, 2007.

\bibitem{MehS11}
V.~Mehrmann and C.~Schr{\"o}der.
\newblock Nonlinear eigenvalue and frequency response problems in industrial
  practice.
\newblock {\em J. Math. Ind.}, 1:18, 2011.

\bibitem{MehV04}
V.~Mehrmann and H.~Voss.
\newblock Nonlinear eigenvalue problems: a challenge for modern eigenvalue
  methods.
\newblock {\em GAMM Mitteilungen}, 27(2):121--152 (2005), 2004.

\bibitem{PerQ22}
J.~P\'{e}rez and M.C. Quintana.
\newblock Linearizations of rational matrices from general representations.
\newblock {\em Linear Algebra Appl.}, 647:89--126, 2022.

\bibitem{Pla82}
J.~Planchard.
\newblock Eigenfrequencies of a tube bundle placed in a confined fluid.
\newblock {\em Comput. Methods Appl. Mech. Engrg.}, 30(1):75--93, 1982.

\bibitem{Ros70}
H.~H. Rosenbrock.
\newblock {\em State Space and Multivariable Theory}.
\newblock Wiley, New York, 1970.

\bibitem{Sch11}
L.~Scholz.
\newblock A derivative array approach for linear second order
  differential-algebraic systems.
\newblock {\em Electron. J. Linear Algebra}, 22:310--347, 2011.

\bibitem{Sol06}
S.~I. Solov'ev.
\newblock Preconditioned iterative methods for a class of nonlinear eigenvalue
  problems.
\newblock {\em Linear Algebra Appl.}, 415(1):210--229, 2006.

\bibitem{SteS90}
G.~W. Stewart and J.-G. Sun.
\newblock {\em Matrix Perturbation Theory}.
\newblock Academic Press, New York, 1990.

\bibitem{SuB11}
Y.~Su and Z.~Bai.
\newblock Solving rational eigenvalue problems via linearization.
\newblock {\em {SIAM} J. Matrix Anal. Appl.}, 32(1):201--216, 2011.

\bibitem{Tis00}
F.~Tisseur.
\newblock Backward error and condition of polynomial eigenvalue problems.
\newblock {\em Linear Algebra Appl.}, 309:339--361, 2000.

\bibitem{Var91}
A.~I.~G. Vardulakis.
\newblock {\em Linear Multivariable Control}.
\newblock Wiley, New York, 1991.

\bibitem{VolA11}
S.~Vologiannidis and E.~N. Antoniou.
\newblock A permuted factors approach for the linearization of polynomial
  matrices.
\newblock {\em Math. Control Signals Systems}, 22(4):317--342, 2011.

\bibitem{Vos03}
H.~Voss.
\newblock A rational spectral problem in fluid-solid vibration.
\newblock {\em Electron. Trans. Numer. Anal.}, 16:93--105 (electronic), 2003.

\end{thebibliography}

\section{Appendix}
So far we have only considered the case that $d_A\geq d_D$ in the polynomial degrees. For completeness in this appendix, we present the analogous results and algorithms for the case $d_{A} < d_{D}$.

\begin{algorithm}[H]
	\caption{Construction of $\mathbb{M}_{\sigma}$ for $\mathbb{L}_{\sigma}(\lam) := \lam \mathbb{M}_{d_D}- \mathbb{M}_{\sigma}$.}
\label{alg2}	
	
	\textbf{Input:} $S(\lambda) = \left[\begin{array}{c|c} \sum_{i=0}^{d_A}\lambda^i A_i & -B\\ \hline C & \sum_{i=0}^{d_D}\lambda^i D_i \end{array}\right]$ and a bijection $\sigma : \{0,1,\ldots ,d_D-1\}\to \{1,2,\ldots,d_D\}$
	
	\textbf{Output:} $\mathbb{M}_{\sigma}$
	
	\begin{algorithmic}
		
		\If{$\sigma$ has a consecution at $0$ }
		\State \scriptsize $\mathbb{W}_0=\left[\begin{array}{cc|cc}
			-A_1&I_n&0 &0\\ -A_0&0&B&0 \\ \hline 0&0&-D_1 &I_p \\ -C & 0&-D_0 & 0
		\end{array}\right]$
		
		\Else
		\State \scriptsize$\mathbb{W}_0 =\left[ \begin{array}{cc|cc} -A_1& -A_0 & 0 & B \\ I_n & 0 &0 &0 \\ \hline 0&-C &-D_1 &-D_0 \\ 0&0&I_m &0 \end{array}\right]$	
		
		\EndIf
		
		If $d_D>d_A$ \\
		\For{ $i = 1:d_A-2$ }
		\If {$\sigma$ has a consecution at $i$ }
		\State {\scriptsize
			$\mathbb{W}_i := \left[\begin{array}{ccc|ccc}-A_{i+1}&I_n&\\ \mathbb{W}_{i-1}(1:i+1,1)&0&\mathbb{W}_{i-1}(1:i+1,2:i+1)&W_{12}\\ \hline 0&0&0& \\ \mathbb{W}_{i-1}(2+i:2i+2,1)&0&\mathbb{W}_{i-1}(2+i:2i+2,2:i+1)&W_{22}\end{array}\right], \text{where}$}
		
		\scriptsize$W_{12}=\left[\begin{array}{ccc}0&0&0\\ \mathbb{W}_{i-1}(1:i+1,i+2)&0&\mathbb{W}_{i-1}(1:i+1,i+3:2i+2)\end{array}\right]$
		
		\scriptsize{$W_{22}=\left[\begin{array}{ccc}-D_{i+1} &I_p&0 \\\mathbb{W}_{i-1}(2+i:2i+2,i+2)&0& \mathbb{W}_{i-1}(2+i:2i+2,i+3:2i+2) \end{array}\right]$}
		\Else 
		\State {\scriptsize$\mathbb{W}_i =\left[\begin{array}{cc|cc} -A_{i+1}& \mathbb{W}_{i-1}(1,1:i+1)& 0&\mathbb{W}_{i-1}(1,2+i:2i+2)\\I_n&0&0&0\\ 0& \mathbb{W}_{i-1}(2:i+1,1:i+1)&0&\mathbb{W}_{i-1}(2:i+1,2+i:2i+2)\\ \hline 0&\mathbb{W}_{i-1}(i+2,1:i+1)&-D_{i+1}&\mathbb{W}_{i-1}(i+2,2+i:2i+2)\\0&0&I_m&0\\0&\mathbb{W}_{i-1}(i+3:2i+2,1:i+1)&0&\mathbb{W}_{i-1}(i+3:2i+2,2+i:2i+2)
			\end{array}\right]$}
		
		\EndIf
		\EndFor
		
		\For {$i=d_A -1:d_D -2 $ }
		\If {$\sigma$ has a consecution at $i$} 
		\State {\tiny
			$\mathbb{W}_i = \left[\begin{array}{c|ccc}\mathbb{W}_{i-1}( 1:d_{A},1:d_{A})&\mathbb{W}_{i-1}(1:d_{A},d_{A}+1)&0&\mathbb{W}_{i-1}(1:d_{A},d_{A}+2:d_{A}+i+1)\\ \hline 0&D_{i+1}&I_p&0\\ \mathbb{W}_{i-1}(d_{A}+1:d_{A}+i+1, 1:d_{A})&\mathbb{W}_{i-1}(d_{A}+1:d_{A}+i+1,d_{A}+1)&0&\mathbb{W}_{i-1}(d_{A}+1:d_{A}+i+1,d_{A}+2:d_{A}+i+1)\end{array}\right]$}
		
		\Else
		\State {\scriptsize$\mathbb{W}_i = \left[\begin{array}{c|cc}\mathbb{W}_{i-1}( 1:d_{A},1:d_{A}) & 0 & \mathbb{W}_{i-1}(1:d_{A},d_{A}+1:d_{A}+i+1) \\ \hline \mathbb{W}_{i-1}(d_{A}+1, 1:d_{A})&-D_{i+1}&\mathbb{W}_{i-1}(d_{A}+1,d_{A}+1:d_{A}+i+1)\\
		0 & I_m& 0\\ \mathbb{W}_{i-1}(d_{A}+2:d_{A}+i+1,1:d_{A}) & 0 & \mathbb{W}_{i-1}(d_{A}+2:d_{A}+i+1,d_{A}+1:d_{A}+i+1)\end{array}\right]$} \\
		\EndIf
		
		\EndFor
		
		\State $\mathbb{M}_{\sigma} := \mathbb{W}_{d_{D}-2}$
	\end{algorithmic}
\end{algorithm}

Based on the construction in Algorithm~\ref{alg2} we have the following.

\begin{lemma}\label{cowii}
	Let $S(\lam)$ be as in (\ref{rosmatrix}) with $A(\lam) = \sum\limits_{i=0}^{d_A}\lam^{i}A_i\in \mathbb C[\lam]^{n\times n}$, $D(\lam) = \sum \limits_{i=0}^{d_D}\lam^{i}D_{i}\in \mathbb C[\lam]^{p\times m}$ and suppose that $d_A< d_D$.  For a bijection $\sig$ Algorithm~\ref{alg2} constructs a sequence of matrices $\{\mathbb{W}_0, \mathbb{W}_1, \ldots, \mathbb{W}_{d_{D}-2}\}, $ where each matrix $\mathbb{W}_i$ for $i = 1, 2, \ldots, d_{D}-2$ is partitioned into blocks in such a way that the blocks of $\mathbb{W}_{i-1}$ are blocks of $\mathbb{W}_i. $

Then
\begin{itemize}
	\item[(a)] The size of $\mathbb{W}_i$ for $i = 0:d_{A}-2$ is $\left[\left(n+n c(\sig(0:i)) + n \mathfrak{i}(\sig(0:i))\right) + \left(p+p c(\sig(0:i)) + m \mathfrak{i}(\sig(0:i))\right)\right]  \times \left[\left(n + n c(\sig(0:i)) + n \mathfrak{i}(\sig(0:i))\right) + \left(m + p c(\sig(0:i)) + m \mathfrak{i}(\sig(0:i))\right)\right] $ and for $i = d_{A}-1: d_{D}-2$ is  $\left[d_{A}n + \left(p+p c(\sig(0:i)) + m \mathfrak{i}(\sig(0:i))\right)\right]  \times \left[d_{A}n + \left(m + p c(\sig(0:i)) + m \mathfrak{i}(\sig(0:i))\right)\right]. $
	
	\item[(b)] The $(1,1)$ diagonal block of $\mathbb{W}_{i}$ is $-A_{i+1}\in\mathbb{C}^{n\times n}$ and the $(3+i,3+i)$ block of $\mathbb{W}_{i}$ is $-D_{i+1}\in\mathbb{C}^{p\times m}$  for $i=0:d_{A}-2$ and the $(1,1)$ block of $\mathbb{W}_{i}$ is $-A_{d_{A}-1}\in\mathbb{C}^{n\times n}$ and the $(d_{A}+1,d_{A}+1)$ block of $\mathbb{W}_{i}$ is $-D_{i+1}\in\mathbb{C}^{p\times m}$ for $i=d_{A}-1:d_{D}-2$. The rest of the diagonal blocks of  $\mathbb{W}_i$  are square zero matrices, and more precisely, for $j = 0, 1, \ldots, i$ and $i= 0, 1, \ldots d_{{A}}-2,$
	$\mathbb{W}_i(i+2-j, i+2-j) = 0_n $ and for $j = 0, 1, \ldots, i$ and $i= 0, 1, \ldots d_{{A}}-2,$
$$\mathbb{W}_i(2i+4-j, 2i+4-j) =
\begin{cases}
0_p, & \text{if } \sig \text{ has a consecution at } j \\
0_m, &  \text{if } \sig \text{ has an inversion at } j
\end{cases}, $$ and  for $j=0, 1, \ldots, i$ and $i=d_{A}-1:d_{D}-2$,
    $$\mathbb{W}_{i}(d_{A}+i+2-j,d_{A}+i+2-j)=\begin{cases}
                                     0_p, \,\,\,\,\ &\text{if $\sigma$ has a conseqution at $j$} \\
                                     0_m, &\text{if $\sigma$ has an inversion at $j$}
                                     \end{cases}.$$
    %
	
	\item[(c)] If $\sig$ has a consecution at $i$, then the size of the zero block in $W_{21}$ block of $\mathbb{W}_i$ is $0_{p \times n}$ and if $\sig$ has an inversion at $i$, then the size of the zero block in $W_{12}$ block of $\mathbb{W}_i$  is $0_{n \times m}$.
\end{itemize}
\end{lemma}
\proof
The proof is analogous to the proof of Lemma~\ref{cowi}.
\eproof

\begin{algorithm}[H] 
\caption{The following algorithm constructs a sequence of matrix polynomials 
$\{\mathbb{N}_{0},\mathbb{N}_{1},\ldots,\mathbb{N}_{d_{D}-2}\}$, where each matrix $\mathbb{N}_{i}$, for $i=1,2,\ldots,d_{D}-2$ is partitioned into blocks in such a way that the blocks of $\mathbb{N}_{i-1}$ are blocks of $\mathbb{N}_{i}$. Let $P_{k}(\lambda)$ and $Q_{k}(\lambda)$ denote the degree $k$ Horner shifts of $A(\lambda)$ and $D(\lambda)$ respectively. For simplicity we drop $\lambda$ from $P_{k}(\lambda)$ and $Q_{k}(\lambda)$ in $\{\mathbb{N}_{i}\}_{i=0}^{d_{D}-2}$. }
\label{alg5}
	 \begin{algorithmic}
		\If {$\sigma$ has a consecution at $0$, }
		\State $\mathbb{N}_0 = \left[
		\begin{array}{cc|cc}
			I_n & 0 & & \\
			\lambda I_n & I_n & & \\
			\hline
			&  & I_p & 0 \\
			&  & \lambda I_p & I_p   \\
		\end{array}
		\right]$
		\Else
		\State $\mathbb{N}_0 = \left[
		\begin{array}{cc|cc}
			0 & -I_n & & \\
			I_n & P_{d_{A}-1} & & \\
			\hline
			&  & 0 & -I_m  \\
			&  & I_p & Q_{d_{D}-1}   \\
		\end{array}
		\right]
		$
		\EndIf
		
		\For {$i=1:d_{A}-2$}
		\If{$\sigma$ has a consecution at $i$}
		\State $\mathbb{N}_i=\left[
		\begin{array}{c|c}
			N_i & \\ \hline
			& N_i'
		\end{array}
		\right]$, where\\
  
		$N_i = \left[ \begin{array}{cc} I_n & 0 \\
			\lambda \mathbb{N}_{i-1}(1:i+1,1) & \mathbb{N}_{i-1}(1:i+1,1:i+1) \end{array} \right]$, \\
		$N_i' = \left[ \begin{array}{cc} I_p & 0 \\           \lambda\mathbb{N}_{i-1}(2+i:2i+2,i+2) & \mathbb{N}_{i-1}(2+i:2i+2,2+i:2i+2) \end{array} \right]$ \\
		
		\Else
		\State $\mathbb{N}_i = \left[
		\begin{array}{c|c}
			N_i & \\ \hline & N_i'
		\end{array}
		\right]$,  where \\
		
		{\scriptsize $N_i = \left[ \begin{array}{ccc} 0 & -I_n & 0 \\ \mathbb{N}_{i-1}(1:i+1,1) & \mathbb{N}_{i-1}(1:i+1,1)P_{d_A-i-1} & \mathbb{N}_{i-1}(1:i+1,2:i+1) \end{array} \right] $, \\
			$N_i' = \left[ \begin{array}{ccc}  0 & -I_m & 0 \\  \mathbb{N}_{i-1}(2+i:2i+2,i+2) & \mathbb{N}_{i-1}(2+i:2i+2,i+2)P_{d_D-i-1} & \mathbb{N}_{i-1}(2+i:2i+2,3+i:2i+2) \end{array} \right]$}

		\EndIf
		\EndFor
		
		\For {$i = d_{A}-1:d_{D}-2$}
		\If{$\sigma$ has a consecution at $i$}
		\State {\scriptsize $\mathbb{N}_i = \left[
		\begin{array}{ccc}
			\mathbb{N}_{i}(1:d_{A},1:d_{A}) & 0 & 0 \\
			0 & I_{p} & 0 \\ 0 & \lambda \mathbb{N}_{i-1}(d_{A}+1:d_{A}+i+1,d_{A}+1) & \mathbb{N}_{i-1}(d_{A}+1:d_{A}+i+1,d_{A}+1:d_{A}+i+1)
		\end{array}
		\right]$}
		\Else 
		\State $\mathbb{N}_{i}=\left[\begin{array}{c|c}
                      \mathbb{N}_{i-1}(1:d_{A},1:d_{A}) &  \\\hline
                        & N_{i}'
                        \end{array}\right]$, where
$$N_{i}' = {\tiny \left[
		\begin{array}{ccc}
			 0 & -I_m & 0 \\
			 \mathbb{N}_{i-1}(d_{A}+1:d_{A}+i+1,d_{A}+1) & \mathbb{N}_{i-1}(d_{A}+1:d_{A}+i+1,d_{A}+1)Q_{d_{D}-i-1} & \mathbb{N}_{i}(d_{A}+1:d_{A}+i+2,d_{A}+2:d_{A}+i+2)
		\end{array}
		\right]}$$
  
		
		\EndIf
		\EndFor
	 \end{algorithmic}
\end{algorithm}

\begin{algorithm}[H] 
\caption{The following algorithm constructs a sequence of matrix polynomials $\{\mathbb{H}_{0},\mathbb{H}_{1},\ldots,\mathbb{H}_{d_{D}-2}\}$, where each matrix $\mathbb{H}_{i}$, for $i=1,2,\ldots,d_{D}-2$ is partitioned into blocks in such a way that the blocks of $\mathbb{H}_{i-1}$ are blocks of $\mathbb{H}_{i}$. Let $P_{k}(\lambda)$ and $Q_{k}(\lambda)$ denote the degree $k$ Horner shifts of $A(\lambda)$ and $D(\lambda)$ respectively. For simplicity we drop $\lambda$ from $P_{k}(\lambda)$ and $Q_{k}(\lambda)$ in $\{\mathbb{H}_{i}\}_{i=0}^{d_{D}-2}$. }
\label{alg6}
  \begin{algorithmic}
  	\If{$\sigma$ has a consecution at $0$}
  	\State $\mathbb{H}_0 := \left[
  	\begin{array}{cc|cc}
  		0 & I_n & & \\
  		-I_n & P_{d_{A}-1} & & \\
  		\hline
  		&  & 0 & I_m \\
  		& & -I_p & Q_{d_{D}-1}  \\
  	\end{array}
  	\right]$
  	
  	\Else
  	\State $\mathbb{H}_0 := \left[
  	\begin{array}{cc|cc}
  		I_{n} & \lambda I_n & & \\
  		0 & I_n &  & \\
  		\hline
  		& & I_{m} & \lambda I_m   \\
  		& & 0 & I_m \\
  	\end{array}
  	\right]$
  	\EndIf
  	
  	\For{$i=1:d_{A}-2$}
  	\If{$\sigma$ has a conseqution at $i$}
  	\State {\scriptsize $\mathbb{H}_i := \left[
  	\begin{array}{cc|cc}
  		0 & \mathbb{H}_{i-1}(1,1:i+1) & & \\
  		-I_n & P_{d_A-i-1}\mathbb{H}_{i-1}(1,1:i+1) & & \\ 
  		0 & \mathbb{H}_{i-1}(2:i+1,1:i+1) & & \\ \hline
  		& & 0 & \mathbb{H}_{i-1}(i+2,2+i:2i+2) \\ 
  		& & -I_p & Q_{d_D-i-1}\mathbb{H}_{i-1}(i+2,2+i:2i+2) \\
  		& & 0 & \mathbb{H}_{i-1}(3+i:2i+2,2+i:2i+2)
  	\end{array}
  	\right]$}
  	
  	\Else
  	\State {\scriptsize $\mathbb{H}_i := \left[
  	\begin{array}{cc|cc}
  		I_n & \lambda\mathbb{H}_{i-1}(1,1:i+1) & & \\ 
  		0 & \mathbb{H}_{i-1}(1:i+1,1:i+1) & & \\ \hline
  		& & I_m & \lambda \mathbb{H}_{i-1}(i+2,i+2:2i+2)\\
  		& & 0 & \mathbb{H}_{i-1}(i+2:2i+2,i+2:2i+2)
  	\end{array}	 
  	\right]$}
  	\EndIf
  	\EndFor

  	\For{$i = d_{A}-1:d_{D}-2$}
  	
  	\If{$\sigma$ has a consecution at $i$}
  	\State {\scriptsize $\mathbb{H}_i := \left[
  	\begin{array}{ccc}
  		\mathbb{H}_{i-1}(1:d_{A},1:d_{A}) & 0 & 0   \\
  		0 & 0 & \mathbb{H}_{i-1}(d_{A}+1,d_{A}+1:d_{A}+i+1)    \\
  		0 & 0 & Q_{d_{D}-i-1}\mathbb{H}_{i-1}(d_{A}+1,d_{A}+1:d_{A}+i+1)  \\
  		0 & 0 & \mathbb{H}_{i-1}(d_{A}+2:d_{A}+i+1,d_{A}+2:d_{A}+i+1) 
  	\end{array}
  	\right]$}
  	
  	\Else
  	\State {\scriptsize  $\mathbb{H}_i :=  \left[
  	\begin{array}{ccc}
  		\mathbb{H}_{i-1}(1:d_{A},1:d_{A}) & 0 & 0  \\
  		0 & I_m & \lambda \mathbb{H}_{i-1}(d_{A}+1,d_{A}+1:d_{A}+i+1)   \\
  		0 & 0 & \mathbb{H}_{i-1}(d_{A}+1:d_{A}+i+1,d_{A}+1:d_{A}+i+1)
  	\end{array}
  	\right]$}
  	\EndIf
  	\EndFor
  \end{algorithmic}
\end{algorithm}
We have the following properties.
\begin{lemma}\label{cohianii}
	Let $S(\lam)$ be as in (\ref{rosmatrix}) with $A(\lam) = \sum\limits_{i=0}^{d_A}\lam^{i}A_i\in \mathbb C[\lam]^{n\times n}$, $D(\lam) = \sum \limits_{i=0}^{d_D}\lam^{i}D_{i}\in \mathbb C[\lam]^{p\times m}$ and suppose that $d_A< d_D$.  Let $\sig: \{0, 1, \ldots, d_{D}-1\} \rightarrow \{1, 2, \ldots, d_{D}\}$ be a bijection and consider 
Algorithms~\ref{alg5} and  \ref{alg6}. Let $\{\mathbb{N}_0, \mathbb{N}_1, \ldots, \mathbb{N}_{d_{D}-2}\} $ and $\{\mathbb{H}_0, \mathbb{H}_1, \ldots, \mathbb{H}_{d_{D}-2}\}$ be given by Algorithms~\ref{alg5} and  \ref{alg6}, respectively.
Consider the sequence of block partitioned matrices $\{\mathbb{W}_i\}_{i=0}^{d_{D}-2}$ constructed by the Algorithm~\ref{alg2}. Then we have the following.
\begin{itemize}
\item[(a)] For $0\leq i \leq d_{D}-2$, and $1 \leq j\leq i+2, $ the number of columns of $\mathbb{N}_i(:, j)$ is equal to the number of rows of $\mathbb{W}_i(j, :)$ so that the product $\mathbb{N}_i(:, j)\mathbb{W}_i(j, :)$ is well defined.

\item[(b)] For $0\leq i \leq d_{D}-2$, and $1 \leq j\leq i+2, $ the number of columns of $\mathbb{W}_i(:, j)$ is equal to the number of rows of $\mathbb{H}_i(j, :) $ so that the product $\mathbb{W}_i(:, j)\mathbb{H}_i(j, :)$ is well defined.

\item[(c)] For $i= 0:d_{A}-2$, the size of $\mathbb{N}_i$ is $\left[\left(n + n c(\sig(0:i)) + n \mathfrak{i}(\sig(0:i)) \right) +\left(p + p c(\sig(0:i)) + m \mathfrak{i}(\sig(0:i))\right) \right] \times \left[\left(n + n c(\sig(0:i)) + n \mathfrak{i}(\sig(0:i)) \right)+\left(p + p c(\sig(0:i)) + m \mathfrak{i}(\sig(0:i)) \right)\right]$ and  for $i= d_{A}-1: d_{D}-2$ the size of $\mathbb{N}_i$ is
\begin{small}
\begin{eqnarray*}
&&\left[\left(n d_{A}\right) + \left(p+p c(\sig(0:i)) + m \mathfrak{i}(\sig(0:i))\right)\right]  \times \\
&&
\left[\left(n d_{A}\right) + \left(p + p c(\sig(0:i)) + m \mathfrak{i}(\sig(0:i))\right)\right].
\end{eqnarray*}
\end{small}

\item[(d)] For $i= 0:d_{A}-2$, the size of $\mathbb{H}_i$ is $\left[\left(n + n c(\sig(0:i)) + n \mathfrak{i}(\sig(0:i))\right) + \left(m + p c(\sig(0:i)) + m \mathfrak{i}(\sig(0:i))\right)\right] \times \left[\left(n + n c(\sig(0:i)) + n \mathfrak{i}(\sig(0:i)) \right) + \left(m + p c(\sig(0:i)) + m \mathfrak{i}(\sig(0:i)) \right)\right]$ and  for $i= d_{A}-1: d_{D}-2$ the size of $\mathbb{H}_i$ is
\begin{small}
\begin{eqnarray*}
&&\left[\left(n d_{A}\right) + \left(m+p c(\sig(0:i)) + m \mathfrak{i}(\sig(0:i))\right)\right]  \times \\
&&
\left[\left(n d_{A}\right) + \left(m + p c(\sig(0:i)) + m \mathfrak{i}(\sig(0:i))\right)\right].
\end{eqnarray*}
\end{small}
 \item[(e)] The matrix polynomials $\mathbb{N}_i$ and $\mathbb{H}_i$ are unimodular with $\det(\mathbb{N}_i) = \pm 1$ and $\det(\mathbb{H}_i) = \pm 1$.
\end{itemize}
\end{lemma}

\begin{lemma}\label{rbnihiwii}
Let $S(\lam)$ be an $(n+p) \times (n+m)$ RSMP as in (\ref{rosmatrix}) with $A(\lam) = \sum\limits_{i=0}^{d_A}\lam^{i}A_i\in \C[\lam]^{n\times n}$ and $D(\lam) = \sum \limits_{i=0}^{d_D}\lam^{i}D_{i} \in \C[\lam]^{p\times m}$ and $d_A < d_D$. Let $\sig  : \{0, 1, \ldots, d_{D}-1\} \rightarrow \{1, 2, \ldots, d_{D}\}$ be a bijection. Let $\{\mathbb{W}_i\}_{i=0}^{d_{D}-2}, \{\mathbb{N}_i\}_{i=0}^{d_{D}-2}$ and $\{\mathbb{H}_i\}_{i=0}^{d_{D}-2}$ be the sequences of block matrices constructed in Algorithms \ref{alg2}, \ref{alg5}, and \ref{alg6}, respectively. Also consider the numbers $\alpha_{i} = p c(\sig(0:i)) + m \mathfrak{i}(\sig(0:i))$, $\beta_{i} = p c(\sig(i)) + m \mathfrak{i}(\sig(i))$ and $\alpha_{i}' = n c(\sig(0:i)) + n \mathfrak{i}(\sig(0:i))$, $\beta_{i}' = n c(\sig(i)) + n \mathfrak{i}(\sig(i))$. Note that $\beta_{i} = p$ if $\sig$ has a consecution at $i$ and $\beta_{i} = m$, if $\sig$ has an inversion at $i$, for $i = 0, 1, 2, \ldots, d_{D}-2$. Then we have the following.
\begin{itemize}
\item[(a)] For $1 \leq i \leq d_{D}-2$ we have
\[
\mathbb{N}_i \left(\left[
              \begin{array}{cc|cc}
                \lam P_{d_{A}-i-2} &  & & \\
                 & \lam I_{\alpha_i'} & & \\
                \hline
                 &  & \lam Q_{d_{D}-i-2}& \\
                 &  &  & \lam I_{\alpha_i} \\
              \end{array}
            \right]
 -\mathbb{W}_i \right)\mathbb{H}_i 
 \]
 {\tiny \[  = \left[
                      \begin{array}{cc|cc}
                         I_{\beta_i}  &  & 0 & 0  \\
                         & N_{i-1} \left( \left[
                          \begin{array}{cc}
                            \lam P_{d_{A}-i-1} &  \\
                             & \lam I_{\alpha'_{i-1}} \\
                          \end{array}
                        \right] -W_{i-1} \right)H_{i-1} & 0 & -B \\
                        \hline
                         0 & 0  &  I_{\beta_i}  & \\
                        0 &  C  &  &
                         N_{i-1}' \left( \left[
                          \begin{array}{cc}
                            \lam Q_{d_{D}-i-1} &  \\
                             & \lam I_{\alpha_{i-1}} \\
                              \end{array}
                        \right]
                        -W_{i-1}' \right)H_{i-1}'  \\
                      \end{array}
                    \right]\]}

 \item[(b)] For $i = 0$
{\scriptsize \[
\mathbb{N}_0 \left(\left[
              \begin{array}{cc|cc}
                \lam P_{d_{A}-2} &  & &  \\
                 & \lam I_{\alpha_0'} &  &  \\
                \hline
                  &  & \lam Q_{d_{D}-2} &  \\
                 &  &  & \lam I_{\alpha_0}  \\
              \end{array}
            \right]
 -\mathbb{W}_0 \right)\mathbb{H}_0 = \left[
                     \begin{array}{cc|cc}
                       I_{\beta_0'} &  & 0 & 0  \\
                        & A(\lam) & 0 & -B \\
                       \hline
                       0 & 0 & I_{\beta_0} &  \\
                       0 & C & & D(\lam) \\
                     \end{array}
                   \right]
\]}
\end{itemize}
\end{lemma}

\proof
The proof is analogous to the proof of Lemma~\ref{rbnihiwi}.
\eproof

\begin{theorem}
Let $S(\lam)$ be an $(n+p) \times (n+m)$ system matrix given  in (\ref{rosmatrix}) with $A(\lam) = \sum\limits_{i=0}^{d_A}\lam^{i}A_i\in \C[\lam]^{n\times n}$ and $D(\lam) = \sum \limits_{i=0}^{d_D}\lam^{i}D_{i} \in \C[\lam]^{p\times m}$ with $d_A < d_D$.  Let $\sig  : \{0, 1, \ldots, d_{D}-1\} \rightarrow \{1, 2, \ldots, d_{D}\}$ be a bijection. Then any Fiedler pencil $\mathbb{L}_{\sig}(\lam)$ is a linearization for $S(\lam)$ associated with any bijection $\sig.$
\end{theorem}

\proof
The proof is analogous to the proof of Theorem~\ref{LinT}.
\eproof

\end{document}